\documentclass[12pt,reqno]{amsart}
\usepackage{fullpage}
\usepackage{graphicx} 
\usepackage[all]{xy}
\usepackage{bbm}
\usepackage{tikz}
\usepackage{paralist}
\usepackage{mathrsfs}
\usepackage{amsmath}
\usepackage{amsfonts}
\usepackage{tabularx}
\usepackage{amssymb}
\usepackage{amsthm}
\usepackage{amscd}
\usepackage{latexsym}
\usepackage{txfonts}
\usepackage{hyperref}
\usepackage{xcolor}
\theoremstyle{plain}
\newtheorem{theorem}{Theorem}
\newtheorem{conjecture}{Conjecture}
\newtheorem*{conjecture*}{Conjecture}
\newtheorem{lemma}[theorem]{Lemma}
\newtheorem{cor}[theorem]{Corollary}
\newtheorem{rem}[theorem]{Remark}

\newtheorem{proposition}[theorem]{Proposition}

\newtheorem*{theorem*}{Theorem}
\theoremstyle{remark} 
\theoremstyle{remark}

\newcommand\numberthis{\addtocounter{equation}{1}\tag{\theequation}}

\def\F{\mathbb{F}}
\def\C{\mathcal{C}}

\newcommand\Osq{\mathbin{\text{\scalebox{1.1}{$\square$}}}}

\title{Minimum number of distinct eigenvalues of distance-regular and signed Johnson graphs}
\author[]{Shaun Fallat \and Himanshu Gupta \and Allen Herman \and Johnna Parenteau}

\address[S.~Fallat]{Department of Mathematics and Statistics$,$ College West 307$.$14, 3737 Wascana Parkway$,$ University of Regina$,$ SK S$4$S $0$A$2,$ Canada}
\email{shaun.fallat@uregina.ca}

\address[H.~Gupta]{Department of Mathematics and Statistics$,$ College West 307$.$14, 3737 Wascana Parkway$,$ University of Regina$,$ SK S$4$S $0$A$2,$ Canada}
\email{himanshu.gupta@uregina.ca}

\address[A.~Herman]{Department of Mathematics and Statistics$,$ College West 307$.$14, 3737 Wascana Parkway$,$ University of Regina$,$ SK S$4$S $0$A$2,$ Canada}
\email{allen.herman@uregina.ca}

\address[J.~Parenteau]{Department of Mathematics and Statistics$,$ College West 307$.$14, 3737 Wascana Parkway$,$ University of Regina$,$ SK S$4$S $0$A$2,$ Canada}
\email{parenjoh@uregina.ca}
\date{}

\begin{document}
\begin{abstract}
We study the minimum number of distinct eigenvalues over a collection of matrices associated with a graph. Lower bounds are derived based on the existence or non-existence of certain cycle(s) in a graph. A key result proves that every Johnson graph has a signed variant with exactly two distinct eigenvalues. We also explore applications to weighing matrices, linear ternary codes, tight frames, and compute the minimum rank of Johnson graphs. Further results involve the minimum number of distinct eigenvalues for graphs in association schemes, distance-regular graphs, and Hamming graphs. We also draw some connections with simplicial complexes and higher-order Laplacians.
\end{abstract}


\maketitle

\section{Introduction}
Let $G$ be a graph with vertex set $V(G)$ and edge set $E(G)$. Typically, we use standard terminology from graph theory (see, for e.g., \cite{bollobas2013modern}). We refer to $n = |V(G)|$ and $m = |E(G)|$ as the order and the size of $G$, respectively. All the graphs we consider are simple (i.e., undirected, without loops and multiple edges). A graph is called \emph{$k$-regular} if every vertex has exactly $k$ neighbors. The \emph{adjacency matrix} $A(G)$ of a graph $G$ has its rows and columns indexed by $V(G)$, where the $(u,v)$-entry is $1$ if there is an edge between $u$ and $v$, and $0$ otherwise. Since $A(G)$ is a real symmetric matrix, all of its eigenvalues are real. These eigenvalues are closely connected to important combinatorial parameters and are useful in various situations (see, for e.g., \cite{brouwer2011spectra,GR2001}). 

A \emph{signature} on a graph $G$ is a function $\sigma: E(G) \to \{1,-1\}$. A \emph{signed graph}, denoted by $\dot{G} = (G,\sigma)$, is a graph $G$ together with a signature $\sigma$. For basic results and open problems in the theory of signed graphs, see \cite{belardo2019open,zaslavsky1982signed} and the references therein. The \emph{adjacency matrix} $A(\dot{G})$ of a signed graph $\dot{G} = (G,\sigma)$ has its rows and columns indexed by $V(G)$, where the $(u,v)$-entry is equal to $\sigma(\{u,v\})$ if there is an edge between $u$ and $v$, and $0$ otherwise. In this paper we refer $A(\dot{G})$ as a \emph{signed adjacency matrix} of the underlying graph $G$. Notice that the conventional adjacency matrix can also be interpreted as a signed adjacency matrix of the graph $G$ with all edges have positive signs. Let $\dot{\mathcal{S}}(G)$ denotes the set of all signed adjacency matrices of a graph $G$. Note that $|\dot{\mathcal{S}}(G)| = 2^{|E(G)|}$. Let $q(M)$ denote the number of distinct eigenvalues of a matrix $M$. We denote the minimum number of distinct eigenvalues of a graph $G$ among all the signed adjacency matrices of graph $G$ as
\begin{align*}
    \dot{q}(G) := \min\{q(M)|M\in \dot{\mathcal{S}}(G)\}.
\end{align*}

An important problem in spectral graph theory is to identify graphs that have an associated matrix with exactly two distinct eigenvalues, where the matrix is chosen from a fixed set of matrices. One instance of this problem is to characterize graphs $G$ such that $\dot{q}(G) = 2$ (see, for e.g., \cite[Problem 3.9]{belardo2019open}). This problem is significant, as demonstrated by Huang's work in \cite{huang2019induced}. Huang constructed a signed adjacency matrix of the $d$-dimensional hypercube whose eigenvalues are $\pm\sqrt{d}$, each with multiplicity $2^{d-1}$, leading to a breakthrough proof of the Sensitivity Conjecture in theoretical computer science. Another noteworthy connection is with equiangular lines. A signed adjacency matrix of a complete graph corresponds to the Seidel matrix of a graph whose edges are the negative edges of the complete graph. Seidel matrices, particularly those with two distinct eigenvalues, have been extensively studied and are linked to equiangular lines (see, for e.g., \cite[Chapter 11]{GR2001}). 

Let $G$ be a graph such that $\dot{q}(G) = 2$. We note a few key results about such graphs. First, $G$ must be a regular graph  (see \cite{ramezani2015signed}). All such graphs $G$ of degree at most $5$ have been characterized in \cite{hou2019signed,stanic2020spectra,stanic2022signed}, and those with order at most $11$ are presented in \cite{stanic2020spectra}. For results concerning if $G$ is triangle-free we refer to \cite{ghasemian2017signed}. Additionally, many other intriguing results and sporadic constructions can be found in \cite{ramezani2022some, stanic2020decomposition,stanic2024linear}.

Consider a significantly larger set of matrices compared to $\dot{\mathcal{S}}(G)$. Let $\mathcal{S}(G)$ denotes the set of all real symmetric $n \times n$ matrices $A = [a_{ij}]$ where $a_{ij} = 0$ if and only if $\{v_i,v_j\}$ is not an edge in $G$, and the diagonal entries $a_{ii}$ can take any value. The \emph{inverse eigenvalue problem} for a graph $G$ (IEPG) asks to determine all possible spectra of matrices in $\mathcal{S}(G)$ (see, for e.g., \cite{barrett2020inverse,hogben2022inverse}). Characterizing the complete spectrum of a given graph is a very challenging problem. Numerous intriguing variants of this problem have been extensively studied. Notable examples include the minimum number of distinct eigenvalues, minimum rank, and multiplicity lists.

The \emph{minimum number of distinct eigenvalues} for a graph $G$ is
\begin{align*}
   q(G) := \min\{q(M)|M \in \mathcal{S}(G)\}.
\end{align*}
It is easy to check that $q(G) = 1$ if and only if $G$ is an empty graph. The parameter $q(G)$ has been studied in \cite{bjorkman2017applications,leal2002minimum,levene2019nordhaus,levene2022orthogonal}. The connected graphs $G$ with $q(G)= n$ or $n-1$ have been characterized (see \cite{barrett2015generalizations}). On the other extreme, the graphs $G$ with $q(G) = 2$ have no forbidden subgraph characterization, as implied by \cite[Theorem 5.2]{ahmadi2013minimum}.  However, it is known that $q(G) = 2$ if and only if there exists $M\in \mathcal{S}(G)$ such that $M^2 = I$, i.e., there is an orthogonal matrix $M$ in $\mathcal{S}(G)$ (see \cite{ahmadi2013minimum}). Thus, studying graphs with $q(G) = 2$ is equivalent to studying zero patterns of $n \times n$ symmetric orthogonal matrices. A connected graph $G$ on $n$ vertices with $q(G) = 2$ has at least $2n-4$ edges (see \cite{barrett2023sparsity}). The regular graphs $G$ with $q(G) = 2$ and of degree at most four are characterized in \cite{barrett2023regular}. More graphs with $q(G) = 2$ have been studied in \cite{barrett2023regular,chen2014undirected}, and \cite{furst2020tight}. Since $\dot{\mathcal{S}}(G) \subset \mathcal{S}(G)$ so $q(G)\leq \dot{q}(G)$. Two key results on $q(G)$ for a general graph $G$ are presented in Section \ref{sec:general_graphs} of this paper. Theorem \ref{thm_about_set_of_cycles} provides lower bounds on $q(G)$ based on the existence of specific cycles in $G$. The proof involves associating polynomials with $G$ and determining whether a solution with all variables non-zero exists. Theorem \ref{thm:triangle_free_non_bipartite_q>2} shows that if $G$ is of odd order, non-bipartite, and triangle-free, then $q(G) \geq 3$. We emphasize that this paper primarily focuses on the real field. However, when it is necessary to consider the complex field, we will use the notations $\mathcal{S}_{\mathbb{C}}(G)$ and $q_{\mathbb{C}}(G)$, etc. 

A graph of order $n$ is said to be {\em strongly-regular} with parameters $(n, k, \lambda, \mu)$ if it is $k$-regular, every pair of adjacent vertices have $\lambda$ common neighbors, and every pair of distinct nonadjacent vertices have $\mu$ common neighbors (see, for e.g., \cite[Chapter 9]{brouwer2011spectra}). A connected graph with diameter $d$ is called {\em distance-regular} if, for any two vertices $u,v$ at distance $i$, there are precisely $c_i$ neighbors of $v$ at distance $i-1$ from $u$ and precisely $b_i$ neighbors at distance $i+1$ from $u$, where $0\leq i \leq d$ (see, for e.g., \cite[Chapter 4]{BCN}). Note that a strongly-regular graph with parameters $(n,k,\lambda,\mu)$, $\mu > 0$ is a distance-regular graph of diameter $2$ with intersection array $\{k,k-1-\lambda;1,\mu\}$, and vice-versa. 

The adjacency matrix $A(G)$ of a graph $G$ with diameter $d$ has at least $d+1$ distinct eigenvalues. For distance-regular graph with diameter $d$, the above bound is tight, that is, it has exactly $d+1$ distinct eigenvalues. This implies the bound $q(G) \leq d+1$ for a distance-regular graph $G$ with diameter $d$. A natural question is to determine the distance-regular graphs for which the above bound is strict or tight. Thus, in the case of a strongly-regular graph, the problem is to determine whether $q(G)$ equals $2$ or $3$, as posed in \cite[Section 8]{ahmadi2013minimum}. The problem of determining the distance-regular graphs with two distinct eigenvalues is posed as \cite[Problem 6.4]{barrett2023regular}.  Our main focus in this paper is to study $q(G)$ for distance-regular graphs. A distance-regular graph is a specific case of the broader concept of association schemes. In Section \ref{sec:assoc_schemes}, we investigate $q(G)$ for graphs within an association scheme. The main result, Theorem \ref{thm:sym_scheme_idem}, provides a criteria for identifying a graph $G$ with $q(G) = 2$ using the idempotents of the Bose-Mesner algebra (see Section \ref{sec:assoc_schemes} for definitions). Additionally, Theorem \ref{thm_q_of_conjuagacy_class_scheme} identifies certain normal Cayley graphs with $q(G) = 2$ based on a theorem of Burnside. In Section \ref{sec:drgs}, we apply these results to derive four propositions regarding distance-regular graphs. Section \ref{sec:small_drg} includes Table \ref{table1}, summarizing $q(G)$ for several small well-known distance-regular graphs.

The two well-known families of distance-regular graphs are the Johnson graph and the Hamming graph. Let $n,d \in \mathbb{N}$ with $n\geq d+1$, and define $[n]:= \{1,2,\ldots,n\}$. The \emph{Johnson graph} $J(n,d)$ is the graph whose vertices are the $d$-subsets of $[n]$, where two subsets are adjacent when their intersection has size $d-1$. For the \emph{Hamming graph} $H(d,n)$, let $n,d \in \mathbb{N}$ with $n\geq 2$, and define $Y_n :=\{0,1,\ldots,n-1\}$. The vertex set of $H(d,n)$ is $Y_n^d$, where two $d$-tuples are adjacent if and only if they differ in exactly one coordinate. The Johnson graph and the Hamming graph have been useful in various areas, including coding theory, design theory, Ramsey theory, and other branches of combinatorics (see, for e.g., \cite{GMbook,Macwilliams1977}). Results regarding $q(G)$ for the Johnson graph are discussed in Section \ref{section_Johnson}, while those for the Hamming graph are presented in Section \ref{section_Hamming}.

Note that the Johnson graph $J(n,2)$ is the line graph of complete graph $K_n$. It is known that there exists a signed adjacency matrix of the line graph of $K_n$ with exactly two distinct eigenvalues (see for e.g., \cite[Theorem 4.3]{stanic2020spectra} and \cite[Proposition 3.1]{ramezani2022some}). In Theorem \ref{Main_thm_1}, we prove that this holds true for all the Johnson graphs that $q(J(n,d)) = \dot{q}(J(n,d)) =2$. We observe that Johnson graphs are linked to certain simplicial complexes. In the theory of simplicial complexes, the higher-order Laplacian operators play a crucial role. In Section \ref{sec:simp_comp}, we include the relevant preliminaries from simplicial complexes and shows how these tools can also be used to establish lower bounds on $q(G)$ for certain graphs derived from a simplicial complex.

Another important problem in the IEPG is determining the minimum rank of a graph. The minimum rank and positive semidefinite minimum rank of a graph $G$ are defined as follows:
\begin{align*}
    mr(G) &:= \min\{\text{rank}(M)|M\in \mathcal{S}(G)\}\\
    mr_+(G) &:= \min\{\text{rank}(M)|M\in \mathcal{S}(G) \text{ and } M \text{ is positive semidefinite}\}.
\end{align*}
Clearly, $mr(G)\leq mr_+(G)$. The parameters $mr(G)$ and $mr_+(G)$ have been studied in various works (see, for e.g., \cite{mr2, mr1, mr5, mr6, mr3, mr4}). In Theorem \ref{main_thm_2}, we prove that for Johnson graphs $J(n,d)$, we have $mr(J(n,d)) = mr_+(J(n,d)) = \binom{n-2}{d-1}$.
Notably, Theorem \ref{main_thm_2} for $d=2$ follows from \cite[Theorem 3.18]{aim2008zero} and \cite[Theorem 3.1.31]{peters2012positive}. 

As mentioned earlier, the construction of a signed adjacency matrix for $H(d, 2)$ in \cite{huang2019induced} shows that $q(H(d, 2)) = 2$ for all $d \geq 2$ (the same construction is also established in \cite[Corollary 6.9]{ahmadi2013minimum}). In Section \ref{section_Hamming}, we offer a new perspective on this construction, focusing on graphs associated with the maximum cliques of a complete $d$-partite graph where each part has size $2$. It follows from \cite[Lemma 3.3]{barrett2023regular} that $q(H(2, n)) = 3$ for all $n \geq 3$. In Corollary \ref{cor:q(hamming)>3}, we demonstrate that the hypercubes $H(d,2)$ are the only Hamming graphs with exactly two distinct eigenvalues, meaning $q(H(d,n))\geq 3$ for all $n\geq 3$ and $d\geq 2$. However, we suspect that  $q(H(d, n)) = d + 1$ holds for all $d\geq 2$ and $n\geq 3$. Accordingly, we propose Conjecture \ref{conj_about_hamming} in Section \ref{section_Hamming}, which, if proven, would support this suspicion using Theorem \ref{thm_hamming_induced_graph}. In Theorem \ref{thm_q=2_for_tensor_of_complete_even_graph} and Theorem \ref{thm_q=2_for_complement_of_H(d,q,d)}, we demonstrate that certain distance-$d$ graphs of $H(d, n)$ and their complements have two distinct eigenvalues. In Theorem \ref{thm_q=2_for_hypercube_scheme}, we prove that the complements of certain hypercubes, along with other specific graphs, have exactly two distinct eigenvalues.

A \emph{weighing matrix} of weight $w$ and order $n$ is a square $n\times n$ matrix $A$ over $\{0,-1,1\}$ such that $AA^T = wI_n$. Two notable special cases include the Hadamard matrices of weight and order $n$, and the conference matrices of weight $n-1$ and order $n$. The weighing matrices have been extensively studied in design theory and coding theory (see, for e.g., \cite{koukouvinos1997weighing}). We provide a construction of a weighing matrix of weight $d^2$ and order $\binom{2d}{d}$ for any $d\geq 2$ (see Proposition \ref{prop_about_Weighing_matrices}). This construction is detailed in Section \ref{sec:ancillary_results}, where we also present several other results that stem from our findings on Johnson graphs. Specifically, we explore connections with weighing matrices in Section \ref{sec:weighing}, linear ternary codes in Section \ref{sec:tern_codes}, tight frame graphs in Section \ref{sec:tight_frames}, and maximum degree of certain induced subgraphs of Johnson graphs in Section \ref{sec:sensitivity}. 

The paper is organized as follows: Section \ref{sec:prelim} presents the preliminary definitions and results necessary for the subsequent sections. In Section \ref{sec:general_graphs}, we discuss results related to general graphs. Section \ref{sec:graphs_special_properties} focuses on graphs with specific properties and is divided into three subsections: Section \ref{sec:assoc_schemes} covers association schemes, Section \ref{sec:drgs} addresses distance-regular graphs, and Section \ref{sec:simp_comp} explores simplicial complexes. Section \ref{section_Johnson} is dedicated to Johnson graphs, where we first present our results and then provide proofs in Section \ref{sec:reg_Johnson}. Section \ref{sec:ancillary_results} discusses various applications of our findings. In Section \ref{section_Hamming} we examine Hamming graphs. Finally, we conclude with a summary and a discussion of related open questions in Section \ref{sec:summary}.

\section{Preliminaries}\label{sec:prelim}
Let $G$ be a graph with $n$ vertices, $m$ edges, and diameter $d$. The shortest path distance between two vertices $u$ and $v$ is denoted by $d_G(u,v)$. For $0\leq j \leq d$, the \emph{distance-$j$ graph} $G_j$ is defined on the same vertex set as $G$, with two vertices $u$ and $v$ being adjacent if and only if $d_G(u,v)=j$. The adjacency matrix of $G_j$, denoted $A_j$, is called the \emph{distance-$j$ matrix} of $G$. Specifically, $A_j(u,v) = 1$ if $d_{G}(u,v) = j$ and $0$ otherwise. Let $m_j$ denote the number of edges in $G_j$ (so $m_1 = m$). Let $\overline{G}$ denote the complement graph of $G$. For two graphs $G$ and $H$ with the same vertex set $V$, the graph $G\cup H$ represents the graph whose vertex set is $V$ and the edge set is $E(G)\cup E(H)$. We denote the complete and the cycle graphs of order $n$ by $K_n$ and $C_n$, respectively. Let $K_{d \times n}$ denote the \emph{complete $d$-partite graph} where each part consists of $n$ vertices. The graph $C_3$ is referred to as the triangle. A graph is said to be \emph{triangle-free} if it contains no $C_3$. A pair of vertices in a graph $G$ are called \emph{antipodal vertices} if they are farthest apart.

The \emph{Cartesian product} $G\Osq H$ of two graphs $G$ and $H$ is a graph with vertex set $V(G\Osq H) = V(G) \times V(H)$, and vertices $(u_1,v_1)$ and $(u_2,v_2)$ are adjacent if and only if either $u_1 = u_2$ and $v_1$ is adjacent to $v_2$ in $H$, or $v_1 = v_2$ and $u_1$ is adjacent to $u_2$ in $G$. The tensor product of graphs $G$ and $H$, denoted by $G \times H$, has the vertex set $V(G\times H) = V(G) \times V(H)$, and vertices $(u_1,v_1)$ and $(u_2,v_2)$ are adjacent if and only if $u_1$ is adjacent to $u_2$ in $G$ and $v_1$ is adjacent to $v_2$ in $H$. 

Let $\dot{G} = (G,\sigma)$ be a signed graph. The \emph{degree} $k(v)$ of a vertex $v$ in $\dot{G}$ is the degree of $v$ in the underlying graph $G$. The \emph{positive degree} $k_+(v)$ is the number of positive edges incident to vertex $v$. Similarly, the \emph{negative degree} $k_{-}(v)$ is the number of negative edges incident to vertex $v$. The graph whose vertex set is same as that of $G$ and whose edge set consists of all the positive (or negative) edges of $\dot{G}$ is denoted by $G_+$ (respectively by $G_{-}$). The signed adjacency matrix $A(\dot{G})$ is a real symmetric matrix so all of its eigenvalues are real. We denote them by $\lambda_1(\dot{G})\geq \lambda_2(\dot{G}) \geq \cdots \geq \lambda_{n}(\dot{G})$. If it is clear from the context, then we simply write $\lambda_i$ instead of $\lambda_i(\dot{G})$. Note for a graph $G$ with at least one edge we have $\lambda_1 > 0$ and $\lambda_{n} < 0$ since the trace of $A(\dot{G})$ is zero. The \emph{index} of $\dot{G}$ is its largest eigenvalue $\lambda_1$. The \emph{spectral radius} $\rho(\dot{G})$ is its largest eigenvalue in absolute value, i.e., $\rho(\dot{G}) := \max\{\lambda_1,-\lambda_{n}\}$. Let $\Delta(G)$ denote the maximum degree of a graph $G$. The following result implies that the maximum degree is always bounded below by the index of any signed adjacency matrix of the graph.
\begin{lemma}(\cite[Lemma 2.3]{huang2019induced})\label{lem:huang_max_degree}
Suppose $G$ is an undirected graph of order $n$, and $A$ is a symmetric matrix
whose entries are from the set $\{-1,0,1\}$, with rows and columns are indexed by $V(G)$, and
whenever $u$ and $v$ are non-adjacent in $G$, $A(u,v) = 0$. Then $\Delta(G) \geq \lambda_1(A)$.
\end{lemma}

A \emph{$2$-lift} $G'$ of graph $G$ is a graph with two vertices (a \emph{fiber}) for each vertex in $V(G)$. Each edge $\{u, v\} \in E(G)$ corresponds to two edges in the $2$-lift: either \emph{parallel edges} $\{(u_0, v_0), (u_1, v_1)\}$ or \emph{cross edges} $\{(u_0, v_1), (u_1, v_0)\}$, where $\{u_0, u_1\}$ and $\{v_0, v_1\}$ are the fibers of $u$ and $v$, respectively. Let us assign  a signature $\sigma$ on $G$ such that $\sigma(\{u, v\}) = 1$ if parallel edges appear in the 2-lift and $\sigma(\{u, v\}) = -1$ if cross edges appear. Conversely, given a signature, we can construct a $2$-lift by including parallel edges for positive edges and cross edges for negative ones. The following is a useful result about the eigenvalues of the adjacency matrix of a $2$-lift (see, e.g., \cite[Lemma 3.1]{bilu2006lifts}). 
\begin{lemma}(\cite[Lemma 3.1]{bilu2006lifts})\label{lem:Bilu_and_Linial}
Let $A(G)$ be the adjacency matrix of a graph $G$, and $A(\dot{G})$ the signed adjacency matrix associated with a $2$-lift $G'$. Then every eigenvalue of $A(G)$ and every eigenvalue of $A(\dot{G})$ are eigenvalues of $A(G')$. Furthermore, the multiplicity of each eigenvalue of $A(G')$ equals the sum of its multiplicities in $A(G)$ and $A(\dot{G})$.
\end{lemma}

We denote the identity matrix of order $s\times s$ by $I_s$, the zero matrix of order $s\times r$ by $O_{s,r}$, and the all one matrix of order $s\times r$ by $J_{s,r}$. If the orders are clear by the context, then we instead write $I$, $J$, and $O$. Let $\text{row}(M)$ and $\text{col}(M)$ denote the row space and the column space of a matrix $M$, respectively. Let $\text{rank}(M)$ and $\text{null}(M)$ denote the rank and nullity of a matrix $M$. The transpose of a matrix $M$ is denoted by $M^T$ and the trace is denoted by $\text{trace}(M)$. \emph{The entrywise product} of matrices $A$ and $B$ is denoted by $A \circ B$. The \emph{Kronecker product} of matrices $A$ and $B$ is denoted by $A \otimes B$. The following classical result is known as \emph{Cauchy’s Interlacing Theorem} (see, e.g., \cite{fisk2005very}). 
\begin{lemma}(Cauchy's interlace Theorem)\label{lem:Cauchy_interlace}
    Let $A$ be a symmetric $n\times n$ matrix, and $B$ be a $r \times r$ principal submatrix of $A$, for some $r < n$. If the eigenvalues of $A$ are $\lambda_1\geq \lambda_2 \geq \cdots \geq \lambda_n$, and the eigenvalues of $B$ are $\mu_1\geq\mu_2 \geq \cdots \geq \mu_r$, then for all $1\leq i \leq r$, $\lambda_i \geq \mu_i \geq \lambda_{i+n-r}$. 
\end{lemma}

An $n\times n$ real  matrix $A$ is called \emph{orthogonal} if and only if $AA^T =  cI_n$ for some positive constant $c\in \mathbb{R}^+$. An $n\times n$ real matrix $A$ is called \emph{orthogonal matrix with zero diagonal} or an $OMZD(n)$, if and only if it orthogonal, its diagonal entries are all zero, and its off-diagonal entries are all nonzero. We note the following two results from \cite{bailey2018orthogonal, bjorkman2017applications}.
\begin{lemma}(\cite[Theorem 3.2]{bailey2018orthogonal})\label{prop_from_Bailey}
    There exists a symmetic $OMZD(n)$ if and only if $n$ is even and $n\neq 4$.  \qed
\end{lemma}
\begin{lemma}(\cite[Proposition 3.9]{bjorkman2017applications})\label{prop_from_Bjorkman}
    Let $G$ and $H$ be connected graphs. Let $A \in \mathcal{S}(G)$ with a zero diagonal and $B \in \mathcal{S}(H)$ with a zero diagonal. Then $A\otimes B \in \mathcal{S}(G\times H)$ with a zero diagonal. \qed
\end{lemma}

The following Proposition \ref{prop:system_of_eqns}, which is proved next, is needed for later use. We begin by reviewing some definitions and introducing the necessary notation. Let $X$ be a finite set of variables, and let $\mathbb{R}[X]$ be the polynomial ring. A \emph{monomial} in $\mathbb{R}[X]$ is a product of variables with non-negative integer exponent. A monomial is called \emph{square-free} if all the variable exponents are at most one. A \emph{perfect-square monomial} is one in which all the exponents are even. Note that $1 \in \mathbb{R}[x]$ is a monomial which is both square-free and perfect-square. Two or more monomials are said to be \emph{coprime} if they do not share any variables with positive exponents. 

Let $F = \{f_1,f_2,\ldots,f_s\} \subset \mathbb{R}[X]$. We denote the set of variables appearing in the polynomials of $F$ by $\Omega(F)$ and the multiplicity of a variable $x \in \Omega(F)$ is denoted by $\mu_F(x)$. Similarly, for a single polynomial $f$, the corresponding notations $\Omega(f)$ and $\mu_f(x)$ represent the set of variables and the multiplicity of $x$ in $f$, respectively. For example, if $F = \{x^2y+xy^2z^3,xyz+3z^2w\}$, then $\Omega(F) = \{x,y,w,z\}$, $\mu(x) = 4$, $\mu_F(y) = 4$, $\mu_F(w) = 1$, and $\mu_{F}(z) = 6$.
The \emph{ideal generated by $F$} is $\langle f_1,f_2,\ldots,f_s\rangle := \{\sum_{i=1}^sh_if_i\ |\ h_i\in \mathbb{R}[X]\}$. Note that any solution to the system of equations $f_i = 0$ for $1\leq i\leq s$ is a root of every polynomial $f$ in the ideal $f\in \langle f_1,f_2,\ldots,f_s\rangle$. We use this fact in the proof of the following proposition. 
\begin{proposition}\label{prop:system_of_eqns}
    Let $F = \{f_1,f_2,\ldots,f_{s}\}$ be a set of polynomials, where each $f_i$ is the sum of two coprime square-free monomials. If $s$ is odd and the multiplicity $\mu_F(x)$ is even for every $x \in \Omega(F)$, then the system of equations $f_i = 0$ for $1\leq i \leq s$ has no solution in which all variables are non-zero. 
\end{proposition}
\begin{proof} We aim to prove that the ideal $\xi = \langle f_1,f_2,\ldots,f_s \rangle$ contains a polynomial that is a product of a monomial with a sum of two perfect-square monomials. Therefore, if a solution to the system $f_i = 0$ for $1\leq i \leq s$ exists with all variables nonzero, then it would not be a root of such a polynomial. This contradiction will establish the desired result. We proceed under the assumption that no smaller subset of $F$ satisfies the hypothesis of the statement. If such a smaller subset existed, the same proof would apply, yielding the same conclusion. 

We apply the following two operations alternately on the polynomials from $F$. Given two polynomials $p_1,p_2 \in \xi $, let $p_i = xq_i + r_i$ for $i=1,2$, where $x$ is a variable and $q_i$, $r_i$ are monomials. We define the new polynomial $p = q_2p_1-q_1p_2 \in \xi$, which simplifies to $p = q_2r_1-q_1r_2$. The multiplicity of $x$ in $p$ satisfies $\mu_p(x) = \mu_{\{p_1,p_2\}}(x) - 2$, while for all other variables $y \in \Omega(\{p_1,p_2\})\setminus\{x\}$, the multiplicity remains unchanged, i.e., $\mu_{p}(y)= \mu_{\{p_1,p_2\}}(y)$. Similarly, if $p_1 = xq_1 - r_1$ and $p_2 = xq_2 + r_2$, we construct $p = q_1 p_2 - q_2 p_1 \in \xi$, which simplifies to $p = q_2 r_1 + q_1 r_2$. The same multiplicity properties hold: $\mu_p(x) = \mu_{\{p_1, p_2\}}(x) - 2$, while $\mu_p(y) = \mu_{\{p_1, p_2\}}(y)$ for all $y \in \Omega(\{p_1, p_2\}) \setminus \{x\}$. Note that the first operation implies a difference of two monomials and the second results in a sum of two monomials. Moreover, if $xq_1 = r_1$, then the second operation is still valid, and if $p_1 = -xq_1+r_1$, then we apply it on $-p_1$.

Start with any polynomial $g_0 = h_0 \in F$. If there is another polynomial sharing a common variable with $g_0$, select one, say $h_1$.  Now, perform the first operation on $g_0$ and $h_1$ to produce a new polynomial $g_1 \in \xi$. In the next iteration, if there exists a polynomial in $F$, other than $h_0$ and $h_1$, that shares a variable with $g_1$, select it, say $h_2$. Now, apply the second operation on $g_1$ and $h_2$ to obtain $g_2 \in \xi$. Repeat these iterations alternatively between the first and the second operation, and ensuring that no previously chosen polynomial from $F$ is reused. Since $F$ contains $s$ polynomials, the procedure must terminate within at most $s-1$ iterations. If it stops after completing the $t$-th iteration, we obtain two sets $G_{t} = \{g_0,g_1,\ldots,g_{t}\} \subseteq \xi$ and $H_{t} = \{h_0,h_1,\ldots,h_{t}\} \subseteq F$, such that the set of variables 
$\Omega(g_t)$ and $\Omega(F\setminus H_t)$ are disjoint. Since each iteration reduces the multiplicity of a variable (if any) by $2$, the multiplicity of a variable $x$ remains even for all $x\in \Omega(g_t) \cup \Omega(F\setminus H_t)$.

If $t$ is odd, we then conclude with a smaller subset $F\setminus H_t$ that still satisfies the hypothesis of the statement, which contradicts our assumption. Therefore, $g_t$ is a sum of two monomials since $t$ must be even. Let $g_t = \prod_{x\in \Omega(g_t)}x^{a(x)} + \prod_{x\in \Omega(g_t)}x^{b(x)}$, where $a(x)$, $b(x)$ are non-negative integers with $a(x)+b(x)$ is even for all $x\in \Omega(g_t)$. Note that $a(x)-c(x)$ and $b(x)-c(x)$ are even for all $x\in \Omega(g_t)$, where $c(x) = \min\{a(x),b(x)\}$. Thus, 
$$g_t = \prod_{x\in \Omega(g_t)}x^{c(x)}\left(\prod_{x\in \Omega(g_t)}x^{a(x)-c(x)}+\prod_{x\in \Omega(g_t)}x^{b(x)-c(x)}\right)$$ 
is the desired polynomial which is contained $\xi$. That is, a polynomial which is a product of a monomial with a sum of two perfect-square monomials. Hence, this completes the proof. \end{proof}

\section{General graphs}\label{sec:general_graphs}
In this section, we focus on a general graph $G$ with diameter $d$. Specifically, our main theorems establish lower bounds on $q(G)$ based on the existence of cycles (see Theorem \ref{thm_about_set_of_cycles}) and the non-existence of the triangle (see Theorem \ref{thm:triangle_free_non_bipartite_q>2}). Let us consider two sets of variables: $X_E = \{x_e : e \in E(G)\}$ for the edges and $X_V = \{x_u : u \in V(G)\}$ for the vertices, resulting in a total of $m + n$ variables. Now, define a matrix $M$ indexed by $V(G)$ along both rows and columns, with entries:
\begin{align*}
   M(u,v) = \begin{cases} 
   x_u &\text{if } u=v,\\
   x_e &\text{if } e=\{u,v\} \in E(G),\\
   0 &\text{otherwise}.
   \end{cases} 
\end{align*}

Let $j =1,2,\ldots,d$ be a fixed number. Then for any $0\leq i< j$, the entries of $M^{i}\circ A_j$ are zero. However, $M^j \circ A_j$ contains non-zero entries, yielding a collection of $m_j$ polynomials over the variables from $X_E$. Each polynomial corresponds to a pair of vertices at distance $j$, with each monomial of it representing a shortest path between the two vertices. Therefore, the total number of monomials in the polynomial equals the number of shortest paths between the two vertices. Moreover, each polynomial is a sum of square-free monomials of degree $j$, all with coefficients equal to $1$. These $m_j$ polynomials are then collected into a set  $\Phi_j(G)$.

The parameter $q(G)$ is also the smallest positive integer such that there exists an assignment for the $m + n$ variables with $M^{q(G)} \in \text{span}\{I, M, \dots, M^{q(G)-1}\}$ and $x_e \neq 0$ for all $x_e \in X_E$. Thus, if the system of equations $f = 0$ for $f\in \Phi_j(G)$ has no solution where $x_e \neq 0$ for all $x_e \in X_E$, then $q(G)\geq j+1$. We use this observation throughout this paper. 

\begin{rem}\label{rem:induced_subgraph_has_no_zero_free_solution} 
Let $G$ be a graph with diameter $d$, and let $1\leq j \leq d$. If there is a unique path of length $j$ between two vertices $u$ and $v$, then $\Phi_j(G)$ contains a monomial. Therefore, as discussed above, $q(G)\geq j+1$. This result was established as \cite[Theorem 3.2]{ahmadi2013minimum} using the same idea, though without relying on the language of polynomials. 
\end{rem}

The following statement is useful and follows immediately. We include it here as a lemma for future reference.  
\begin{lemma}\label{lem:subgraph_no_zero_free_solution}
    If there exists an induced subgraph $H$ of graph $G$ such that $\Phi_j(H) \subseteq \Phi_j(G)$ and the system of equations $f = 0$ for $f\in \Phi_j(H)$ has no solution in which all variables are non-zero for some $j$. Then the system of equations $f = 0$ for $f\in \Phi_j(G)$ has no such solution as well. In particular, $q(G) \geq j+1$. \qed
\end{lemma}

The following result establishes a criterion, based on the existence of induced cycles in a graph $G$, that ensures $q(G)$ is sufficiently large.

\begin{theorem}\label{thm_about_set_of_cycles}
Let $G$ be a graph with diameter $d$, containing $s$ distinct induced copies of the cycle $C_{2j}$ for some $2\leq j \leq d$. If the following conditions hold:
\begin{enumerate}[(i)]
\item for each of the $s$ cycles, there is a pair of antipodal vertices on it which is connected by exactly two shortest paths of length $j$ in $G$;
\item the number of cycles $s$ is odd;
\item each edge of $G$ either does not appear in any of the $s$ cycles or appears an even number of times on them.
\end{enumerate}
Then $q(G) \geq j+1$. 
\end{theorem}
\begin{proof}
Consider the set of polynomials $\Phi_j(G)$. Condition $(i)$ guarantees the existence of a set $F = \{f_1,f_2,\ldots,f_s\} \subseteq \Phi_j(G)$, where each $f_i$ is the sum of two coprime square-free monomials. Condition $(iii)$ ensures that the multiplicity $\mu_F(x)$ is even for every $x \in \Omega(F) \subseteq X_E$. Therefore, by Proposition \ref{prop:system_of_eqns} the system of equations $f = 0$ for $f\in F$ has no solution in which all variables are non-zero. Consequently, the same holds for the system of equations $f = 0$ for $f\in \Phi_j(G)$. As a result, $q(G)\geq j+1$. 
\end{proof}

Our next result is about a triangle-free graph. 

\begin{proposition}\label{thm_zero_diag_for_triangle_free}
Let $G$ be a connected triangle-free graph. If $M \in \mathcal{S}(G)$ such that $M^2 = I$, then the following holds:
\begin{enumerate}
\item[(i)] if $G$ is a bipartite graph with partition sets are $V_1$ and $V_2$, then there exists a number $b$ such that $M_{u,u} = b = -M_{v,v}$ for all $u \in V_1$ and $v \in V_2$;
\item[(ii)] if $G$ is a non-bipartite graph, then $M_{v,v} = 0$ for all $v \in V(G)$. 
\end{enumerate}
\end{proposition}
\begin{proof}
Let $u \in V(G)$ such that $M_{u,u} = b$. Consider the block partition of $M$ according to $\{u\}$, the neighborhood $N(u)$, and the non-neighborhood $N^c(u)$ as
\begin{align*}
M = \begin{bmatrix}
b & U^T & O\\
U & -D & Y^T\\
O & Y & Z
\end{bmatrix}.
\end{align*}
Notice that $O$ is a zero matrix, and $U$ is a nowhere-zero matrix. Moreover, $D$ is a diagonal matrix since $G$ is a triangle-free graph. Now, we have
\begin{align*}
M^2 &= \begin{bmatrix}
b^2+U^TU & (b U - D U)^T & (YU)^T\\
b U - D U & D^2+ UU^T+Y^TY & (ZY-YD)^T\\
YU & ZY-YD & Z^2+YY^T
\end{bmatrix}
 = \begin{bmatrix}
I & O & O\\
O & I & O\\
O & O & I
\end{bmatrix}.
\end{align*}
The $(2,1)$-th positions of the last equation implies that $DU = b U$. However, since $D$ is a diagonal matrix, and $U$ is a nowhere-zero matrix so we must have $D = b I$. Therefore, we have proved that if $\{u,v\}$ is an edge, then $M_{u,u} = -M_{v,v}$. Since $G$ is connected so if $G$ is bipartite, then the statement $(i)$ follows. 
Also, if there exists a vertex $u \in V(G)$ such that $M_{u,u} = 0$, then $M_{v,v} = 0$ for every $v\in V(G)$. Now, let $G$ be a non-bipartite graph. Therefore, we assume that $C_{2\ell+1}$ be an odd cycle that contains in $G$. Now, for any $u \in V(C_{2\ell+1})$ we have $M_{u,u} = 0$. If not, there is a contradiction from the fact that $M_{u,u} = -M_{v,v}$ whenever $\{u,v\}$ is an edge. Hence proved.
\end{proof}

The following theorem implies that if the triangles are forbidden in a non-bipartite graph $G$ of odd order, then $q(G)$ can never be equal to $2$. 
\begin{theorem}\label{thm:triangle_free_non_bipartite_q>2}
Let $G$ be a connected triangle-free graph of odd order $n$. If $G$ is a non-bipartite graph, then $q(G) \geq  3$. 
\end{theorem}
\begin{proof}
Suppose $q(G) = 2$. Then there exists $M\in \mathcal{S}(G)$ such that $M^2 = I$. Thus, the distinct eigenvalues of $M$ are $-1$ and $1$. By Proposition \ref{thm_zero_diag_for_triangle_free} we must have $\text{trace}(M) = 0$. Therefore, the multiplicities of the eigenvalues $-1$ and $1$ of $M$ are equal. However, that is not possible since $n$ is odd. Hence, $q(G) \geq 3$. 
\end{proof}

\section{Graphs with Special Properties}\label{sec:graphs_special_properties}
In this section, we examine the parameter $q(G)$ for graphs exhibiting certain algebraic and regularity properties. Specifically, we focus on association schemes, distance-regular graphs, and simplicial complexes, each covered in the following three sections.

\subsection{Association schemes}\label{sec:assoc_schemes}
We start by introducing the essential definitions and preliminary results on association schemes required for our work (see, for e.g., \cite[Chapter 2]{bannai2021algebraic}).
An \emph{association scheme (commutative) with $d$ classes} is a collection $\mathcal{A} = \{A_0, A_1, \dots, A_d\}$ of $n \times n$ matrices, where each matrix has entries that are either $0$ or $1$. Furthermore, these matrices satisfy the following conditions:
\begin{inparaenum}[(i)]
\item $A_0 = I_n$,
\item $A_0 + A_1 + \dots + A_d = J_n$, \item for all $0 \leq i \leq d$, the transpose $A_i^T \in \mathcal{A}$, and
\item for all $0 \leq i, j \leq d$, the product $A_i A_j  = A_jA_i = \sum_{k=0}^{d} p_{i,j}^k A_k$, where $p_{i,j}^k$ are non-negative integers.
\end{inparaenum} 

If $A_i^T = A_i$ for all $0\leq i \leq d$, the association scheme is called a \emph{symmetric association scheme}. Thus, a symmetric association scheme corresponds to a decomposition of the complete graph $K_n$ into $d$ graphs, whose adjacency matrices $A_1,A_2,\ldots,A_d$, along with the identity matrix $A_0 = I$, satisfy the condition (iv). From this point onward, we refer to an association scheme simply as a scheme. A graph within a scheme is one whose adjacency matrix is given by $A_J := \sum_{j \in J} A_j$ for some subset $J \subseteq \{1,2,\ldots,d\}$. 

The \emph{Bose-Mesner algebra} $\mathbb{C}(\mathcal{A})$ of a scheme $\mathcal{A} = \{A_0,A_1,\ldots,A_d\}$ is the commutative algebra generated by the matrices $A_0,A_1,\ldots,A_d$. Equivalently, by the definition of a scheme, it is the complex linear span of these matrices, and the set $\{A_0,A_1,\ldots,A_d\}$ forms a basis of $\mathbb{C}(\mathcal{A})$. There is also another basis of $\mathbb{C}(\mathcal{A})$, consisting of idempotent matrices. Specifically, there exists a basis $\{E_0,E_1,\ldots,E_d\}$ of $\mathbb{C}(\mathcal{A})$ such that \begin{inparaenum}
    \item $E_0=\frac{1}{n}J$,
    \item for all $0\leq i,j \leq d$, the product $E_iE_j = \delta_{i,j}E_i$, where $\delta_{i,j}$ is the Kronecker delta,
    \item $\sum_{i=0}^{d}E_i = I$, and
    \item for all $0\leq i\leq d$, the transpose $E_{i}^T$ is also in the set $\{E_0,E_1,\ldots,E_d\}$.
\end{inparaenum}
Additionally, the columns of each $E_i$ are eigenvectors for every matrix in $\mathbb{C}(\mathcal{A})$. Since $\{A_0,A_1,\ldots,A_d\}$ and $\{E_0,E_1,\ldots,E_d\}$ represents two distinct bases of $\mathbb{C}(\mathcal{A})$, the following relations hold:
\begin{align*}
    A_j &= \sum_{i=0}^{d}P_{j}(i)E_i, \ \ 0\leq j \leq d,\\
    E_j &= \frac{1}{n}\sum_{i=0}^{d}Q_{j}(i)A_i, \ \ 0\leq j \leq d.
\end{align*}

The change of basis matrices $P = (P_{j}(i))_{0\leq i,j \leq d}$ and $Q = (Q_j(i))_{0\leq i,j \leq d}$ are referred to as \emph{the first eigenmatrix} (or \emph{the character table}) and \emph{the second eigenmatrix} (or \emph{the dual character table}) of the association scheme $\mathcal{A}$, respectively. The rows of $P$ correspond to the irreducible characters of the algebra $\mathbb{C}(\mathcal{A})$. Note that the $j$-th column of the matrix $P$ contains all the eigenvalues of $A_j$ (without accounting for multiplicities). The multiplicity of an eigenvalue $P_j(i)$ of $A_j$ is given by $m_i := Q_{i}(0)$. The eigenvalue $P_j(0)$ is the valency of the regular graph whose adjacency matrix is $A_j$. Furthermore, the primitive idempotent matrices $E_i$ can also be expressed as 
\begin{align}\label{eqn:scheme_idemp}
    E_i = \frac{m_i}{n} \sum_{j=0}^d \frac{\overline{P_{j}(i)}}{P_{j}(0)} A_j, \ \ 0\leq i \leq d,
\end{align}
where $\overline{z}$ denote the complex conjugate of $z$ (see, for e.g., \cite[Theorem 2.22 (3)]{bannai2021algebraic}). Thus, for every subset $I \subseteq \{0,1,\ldots,d\}$ we have
\begin{align*}
    E_I := \sum_{i \in I}E_i = \frac{1}{n}\sum_{i\in I} \sum_{j=0}^d \frac{m_i\overline{P_{j}(i)}}{P_{j}(0)} A_j.
\end{align*}
Therefore, if $A_j(x,y) = 1$ for some $0\leq j\leq d$, then $E_I(x,y) = \frac{1}{n}\sum_{i \in I}\frac{m_i\overline{P_{j}(i)}}{P_{j}(0)}$. For every proper subset $I \subset \{0,1,\ldots,d\}$, the matrix $E_I$ has only two distinct eigenvalues, $0$ and $1$. This gives us a way to identify idempotent matrices in $\mathbb{C}(\mathcal{A})$ that can be used to prove certain graphs in the scheme will have two distinct eigenvalues. More precisely, suppose there exist two subsets $I \subset \{0,1,\ldots,d\}$ and $J \subseteq \{1,2,\ldots,d\}$ such that $\sum_{i \in I}\frac{m_i\overline{P_{j}(i)}}{P_{j}(0)}\neq 0$ if and only if $j \in J\cup \{0\}$. Now, consider the graph $G_J$ whose adjacency matrix is given by $A_J = \sum_{j \in J}A_j$. Then the matrix $E_I$ has two distinct eigenvalues, and its zero and non-zero pattern of off-diagonal entries is same as that of $A_J$. 

Recall that for a symmetric scheme, the matrices $A_j$ are all symmetric. Thus, the eigenvalues $P_j(i)$ are all real. Consequently, the idempotent matrices $E_I$ are real and symmetric for all subsets $I \subseteq \{0,1,\ldots,d\}$. Hence, we obtain the following theorem. 
\begin{theorem}\label{thm:sym_scheme_idem}
    Let $\mathcal{A} = \{A_0,A_1,\ldots,A_d\}$ be a symmetric association scheme with $d$ classes. Suppose there exist two subsets $I \subset \{0,1,\ldots,d\}$ and $J \subseteq \{1,2,\ldots,d\}$ such that $\sum_{i \in I}\frac{m_iP_{j}(i)}{P_{j}(0)}\neq 0$ if and only if $j \in J\cup \{0\}$. Let us consider the graph $G_J$ whose adjacency matrix is given by $A_J = \sum_{j \in J}A_j$. Then $E_I$ is an idempotent contains in $\mathcal{S}(G_J)$. In particular, $q(G_J) = 2$. 
\end{theorem}

Every finite group gives rise to an association scheme. Let $H$ be a finite group, and let $\{C_0 = \{e\}, C_1, \dots, C_d\}$ be the set of conjugacy classes of $H$, where $e$ is the identity element of $H$. For each $0\leq j \leq d$, let $h_j$ be a representative element of the conjugacy class $C_j$. For every $0\leq i\leq d$, we define the matrix $A_i$ with rows and columns indexed by the elements of $H$. Further, the entry $A_i(x, y)$ is equal to $1$ if $y^{-1}x \in C_i$, and $0$ otherwise. 

The collection $\mathcal{A}_H = \{A_0, A_1, \dots, A_d\}$ forms an association scheme with $d$ classes which is known as \emph{the conjugacy class association scheme} (see, for e.g., \cite[Example 2.8]{bannai2021algebraic}). However, this scheme is not necessarily symmetric. A \emph{normal Cayley graph} of a group $H$ is an undirected graph in its conjugacy class scheme. In these graphs, the adjacency structure reflects the group's internal symmetries and conjugation relations.

In the following discussion, we utilize the representation theory of finite groups, particularly their character theory. For more details, see, for e.g., \cite[Chapter 3]{isaacs1976character}. Let $\text{Irr}(H) = \{\chi_0,\chi_1,\ldots,\chi_{d}\}$ denote the set of all irreducible, inequivalent characters of $H$, where $\chi_0$ is the trivial character. A character $\chi$ is called \emph{non-linear} if $\chi(e) > 1$. If $H$ is non-Abelian, it must have at least one non-linear character. Burnside’s following classical theorem further asserts that such irreducible characters must have nontrivial vanishing sets (see, for e.g., \cite[Theorem 3.15]{isaacs1976character}).
\begin{theorem}(Burnside) \label{thm_Burnside} 
Let $H$ be a non-Abelian finite group, and let $\chi \in \text{Irr}(H)$ be a non-linear character. Then there exists $h \in G\setminus\{e\}$ such that $\chi(h)=0$. 
\end{theorem} 

The Bose-Mesner algebra of the scheme $\mathcal{A}_H$ is isomorphic to the center of the group algebra $\mathbb{C}(H)$.  The irreducible characters of the Bose-Mesner algebra of $\mathcal{A}_H$ are in bijection with the irreducible characters of $H$. More precisely, the rows of the first eigenmatrix $P$ correspond to the irreducible characters, while the columns correspond to the conjugacy classes of $H$, and the following holds:
\begin{align*}
    P_{j}(i) = \frac{|C_j|\cdot \chi_i(h_j)}{\chi_i(e)}.
\end{align*}
Additionally, the multiplicities are given by $m_i = \chi_i(e)^2$. Therefore, equation (\ref{eqn:scheme_idemp}) implies that the primitive idempotents of $\mathcal{A}_H$ are given by
\begin{align*}
    E_i = \frac{m_i}{n} \sum_{j=0}^d \frac{\overline{P_{j}(i)}}{P_{j}(0)} A_j = \frac{\chi_i(e)}{|H|}\sum_{j=0}^d \overline{\chi_i(h_j)}A_j, \ \ 0\leq i \leq d. 
\end{align*}

Note that for any $h \in H$ and $\chi \in \text{Irr}(H)$, we have $\chi(h^{-1}) = \overline{\chi(h)}$. Hence, $\chi(h) \neq 0$ if and only if $\chi(h^{-1}) \neq 0$.
In combination with Theorem \ref{thm_Burnside}, this implies that non-linear irreducible characters associated with a conjugacy class scheme must vanish on at least one nontrivial matrix $A_j$. Consequently, the corresponding primitive idempotent (possibly complex Hermitian) corresponds to some nontrivial normal Cayley graph of the group $H$. We summarize this result as the following theorem.

\begin{theorem} \label{thm_q_of_conjuagacy_class_scheme}
Let $\mathcal{A}_H$ be the conjugacy class scheme of a non-Abelian group $H$, and let $\chi_i \in \text{Irr}(H)$ be a non-linear character. Define $J = \{j: \chi_i(h_j) \ne 0\}$ as the complement of the vanishing set of $\chi_i$. Let $G_{J}$ be the normal Cayley graph in the association scheme $\mathcal{A}_H$, with adjacency matrix given by $A_{J} = \sum_{j \in J} A_j$. Then $E_i$ is an idempotent contains in $\mathcal{S}_{\mathbb{C}}(G_J)$. In particular, $q_{\mathbb{C}}(G_{J}) = 2$. 
\end{theorem} 

\begin{rem}
A group is called \emph{ambivalent} if every element is conjugate to its inverse. For a finite group, this is equivalent to stating that all its characters are real-valued. Notable examples of non-Abelian ambivalent groups include the symmetric group $S_n$ and the dihedral group $D_n$ for any $n\geq 3$. We note that if the group $H$ is ambivalent in Theorem \ref{thm_q_of_conjuagacy_class_scheme}, then the conclusion is over the real field instead of the complex field. However, the desired conclusion over the real field can still hold even if the group is not ambivalent.
\end{rem}

\subsection{Distance-regular graphs}\label{sec:drgs}
Recall that a connected graph with diameter $d$ is called {\em distance-regular} if, for any two vertices $u,v$ at distance $i$, there are precisely $c_i$ neighbors of $v$ at distance $i-1$ from $u$ and precisely $b_i$ neighbors at distance $i+1$ from $u$, where $0\leq i \leq d$. The ordered array $\{ b_0, b_1, \ldots, b_{d-1}; c_1, c_2, \ldots, c_d \}$ is called {\it the intersection array} of the distance-regular graph. By definition, a distance-regular graph is a $k$-regular graph with $k=b_0$. Also, for any two vertices $u$ and $v$ at distance $i$, the number of neighbors of $v$ at distance $i$ from $u$ is given by $a_i:= k-b_i-c_i$. A useful fact is that the intersection array of a distance-regular graph is monotonic: $k = b_0 \ge b_1 \ge \dots \ge b_d$ and $1 = c_1 \le c_2 \le \dots \le c_d$. 

The collection $\mathcal{A} = \{A_0,A_1,\ldots,A_d\}$ of distance matrices of a distance-regular graph $G$ with diameter $d$ forms a symmetric association scheme with $d$ classes (see, for e.g., \cite[Section 2.8]{bannai2021algebraic}). By the definition of distance-regular graph, any polynomial $\Phi_j(G)$ is a sum of $\prod_{i=1}^{j}c_i$ square-free monomials (a property which is not true in general). The next result follows immediately from Remark \ref{rem:induced_subgraph_has_no_zero_free_solution} and the above discussion. It has already been noted for strongly-regular graphs in \cite{ahmadi2013minimum}.
\begin{proposition}\label{prop:c_i_are_1}
    Let $G$ be a distance-regular graph with diameter $d$, and let $j$ be the largest integer such that $c_j = 1$. Then $j+1\leq q(G) \leq d+1$. \qed
\end{proposition} 

Suppose that there exists $j$ such that $c_j=2$ and $c_i = 1$ for all $1\leq i\leq j-1$. Then each polynomial in $\Phi_j(G)$ is a sum of two coprime square-free monomials. Furthermore, each polynomial represent an induced cycle of length $2j$ of $G$. Conversely, each cycle of length $2j$ in $G$ corresponds to $j$ distinct polynomials, one for each antipodal pair of the cycle. Note that a distance-regular graph is triangle-free if and only if $a_1 = 0$. Therefore, the following two results follows from Theorem \ref{thm_about_set_of_cycles} and Theorem \ref{thm:triangle_free_non_bipartite_q>2}, respectively.

\begin{proposition}\label{prop:q(G)>t}
    Let $G$ be a distance-regular graph with diameter $d$ such that for some $2\leq j\leq d$ we have $c_j = 2$ and $c_{j-1} = 1$ for all $1\leq i < j$. Suppose $G$ contains $s$ distinct induced copies of the cycle $C_{2j}$ such that each edge of $G$ either does not appear in any of these cycles or appears an even number of times on them. If $s$ is odd, then $q(G)\geq j+1$. \qed
\end{proposition}

\begin{proposition}\label{cor_triangle_free_drg}
Let $G$ be a non-bipartite distance-regular graph of odd order with $a_1 = 0$. Then $q(G) \geq 3$. In particular, if $G$ is a strongly-regular graph of odd order with $\lambda = 0$, then $q(G) = 3$. \qed
\end{proposition}

A connected distance-regular graph $G$ of diameter $d$ is called \emph{antipodal} if its distance $d$ graph $G_d$ is a disjoint union of complete graphs. These complete graphs are called the \emph{fibers} of $G$, and all have the same size. We also say $G$ is an \emph{antipodal $r$-cover}, where $r$ is the size of the cliques of $G_d$. When $r=2$, we say \emph{antipodal double-cover}. Let us consider a graph $\overline{G}$ whose vertex set is the set of fibers of $G$, and two such fibers are adjacent whenever there is an edge between them in graph $G$. The graph $\overline{G}$ is called the \emph{folded graph} of $G$. An antipodal double-cover $G$ is a $2$-lift of $\overline{G}$. We note the following two important results regarding antipodal distance-regular graphs.

\begin{proposition}(\cite[Proposition 4.2.2(ii)]{BCN})\label{prop_BCN_1}
Let $G$ be a distance-regular graph of diameter $d \in \{2m,2m+1\}$ with intersection array 
$
\{b_0,b_1,\ldots,b_{d-1};c_1,c_2,\ldots,c_d\}.
$ Then the graph $G$ is an antipodal $r$-cover if and only if $b_i = c_{d-i}$ for all $i = 0,1,\ldots,d$, except for $i = m$, and $r = 1 + \frac{b_m}{c_{d-m}}$. Moreover, $\overline{G}$ is a distance-regular graph of diameter $m$ with intersection array
\begin{align*}
    \{b_0,b_1,\ldots,b_{m-1};c_1,c_2,\ldots,\gamma c_m\},
\end{align*}
where $\gamma = r$ if $d=2m$, $\gamma = 1$ if $d=2m+1$. 
\end{proposition}

\begin{proposition}(\cite[Proposition 4.2.3(ii)]{BCN})\label{prop_BCN_2}
    Let $G$ be distance-regular with spectrum $\Psi$, where $\theta \in \Psi$ has multiplicity $m(\theta)$. If $G$ is antipodal of diameter $d\geq 3$, then $\overline{G}$ has a spectrum that is a subset $\overline{\Psi}$ of $\Psi$, and for $\theta \in \overline{\Psi}$ the multiplicity of $\theta$ in $\overline{\Psi}$ and $\Psi$ are the same.  
\end{proposition}

Using the above two propositions together with Lemma \ref{lem:Bilu_and_Linial}, we obtain the following result.
\begin{proposition}\label{prop:double_cover_drg}
    Let $G$ be a distance-regular graph with diameter $d$. Suppose there exists a distance-regular graph $G'$ with diameter $2d$ that is an antipodal double cover of $G$. Then $q(G) \leq \dot{q}(G) \leq d$. 
\end{proposition}
\begin{proof}
The adjacency matrix of $G$ has $d+1$ distinct eigenvalues, while that of $G'$ has $2d + 1$. Since $G$ is a $2$-lift of $G'$, by Proposition \ref{prop_BCN_2} and Lemma \ref{lem:Bilu_and_Linial}, there exists a signed adjacency matrix of $G$ with $2d + 1 - (d + 1) = d$ distinct eigenvalues. Thus, the result follows.
\end{proof}

\begin{rem}
In Proposition \ref{prop:double_cover_drg}, we considered only the antipodal double cover since our focus is on the real field. If we consider the complex field, then a similar result can be derived for an antipodal $r$-covers by using the representation theory of finite groups.
\end{rem}

Both the Hamming graph $H(d,2)$ and the Johnson graph $J(2d,d)$ are antipodal double covers for all $d\geq 2$. Their folded graphs are known as the \emph{folded cube} $\overline{H(d,2)}$ and the \emph{folded Johnson graph} $\overline{J(2d,d)}$. The folded cube $\overline{H(d,2)}$ is isomorphic to the graph obtained from $H(d-1,2)$ by joining every pair of vertices at distance $d-1$. By Proposition \ref{prop:double_cover_drg}, if $d$ is even, then $\dot{q}(\overline{H(d,2)}) \leq \frac{d}{2}$, and $\dot{q}(\overline{J(2d,d)})\leq \frac{d}{2}$. However, it follows from \cite[Proposition 1.3]{alon2021unitary} that $\dot{q}(\overline{H(d,2)}) = 2$ when $d \equiv 0,1 \pmod{4}$ (see also \cite[Section 9.2E]{BCN} and \cite[Section 3.3]{iga2023eigenvalues}). Additionally, for $d\equiv 2,3 \pmod{4}$, it follows from \cite[Proposition 1.3]{alon2021unitary} that $\dot{q}_{\mathbb{C}}(\overline{H(d,2)}) = 2$ over the complex field.

\subsubsection{Some small distance-regular graphs}\label{sec:small_drg}
In this section, we investigate the minimum number of distinct eigenvalues of several small distance-regular graphs using the results discussed above. 
For the sake of brevity, detailed calculations are omitted, and only the conclusions are summarized in Table \ref{table1}, with the exception of the following two graphs. 

Let $G$ be the Hamming graph $H(2,3)$, a strongly-regular of order $9$ with intersection array $\{4,2;1,2\}$. The graph $G$ contains $9$ distinct $C_4$ cycles, with each edge in $E(G)$ appearing twice among these cycles (see Figure \ref{fig:H(2,3)}). Therefore, the system of equations $f = 0$ for $f \in \Phi_2(H(2,3))$ has no solution in which all variables are non-zero. By Proposition \ref{prop:q(G)>t}, we have that $q(H(2,3)) = 3$. We use this fact in Section \ref{section_Hamming} to provide a more general result.  

\begin{figure}[ht]
    \centering
\begin{tikzpicture}[scale=1.2, every node/.style={circle, draw, fill=black, inner sep=2pt}]
\node (1-1) at (0,0) {};
\node (1-2) at (1,0) {};
\node (1-3) at (2,0) {};

\node (2-1) at (0,1) {};
\node (2-2) at (1,1) {};
\node (2-3) at (2,1) {};

\node (3-1) at (0,2) {};
\node (3-2) at (1,2) {};
\node (3-3) at (2,2) {};

\foreach \i in {1,2,3} {
  \draw[line width=0.5mm] (\i-1) -- (\i-2);
  \draw[line width=0.5mm] (\i-2) -- (\i-3);
}

\foreach \j in {1,2,3} {
  \draw[line width=0.5mm] (1-\j) -- (2-\j);
 \draw[line width=0.5mm] (2-\j) -- (3-\j);
}

\draw[bend right,line width=0.5mm] (0,0) to (2,0);
\draw[bend left,line width=0.5mm] (0,0) to (0,2);
\draw[bend right,line width=0.5mm] (2,2) to (0,2);
\draw[bend right,line width=0.5mm] (2,0) to (2,2);
\draw[bend left,line width=0.5mm] (1,0) to (1,2);
\draw[bend left,line width=0.5mm] (0,1) to (2,1);

\end{tikzpicture}
\hspace{1.5cm}
\begin{tikzpicture}[scale=1.2, every node/.style={circle, draw, fill=black, inner sep=2pt}]

\draw[dashed,bend right, draw=red, thick] (1,1) to (1,2);
\draw[dashed,bend left, draw=red, thick] (1,1) to (2,1);
\draw[dashed,bend left, draw=red, thick] (2,2) to (1,2);
\draw[dashed,bend left, draw=red, thick] (2,1) to (2,2);

\draw[dotted,bend right, draw=blue, thick] (0,0) to (0,1);
\draw[dotted,bend left, draw=blue, thick] (0,0) to (1,0);
\draw[dotted,bend left, draw=blue, thick] (1,1) to (0,1);
\draw[dotted,bend left, draw=blue, thick] (1,0) to (1,1);

\draw[bend right, draw=black, thick,dash pattern=on 4pt off 2pt on 1pt off 2pt] (1,0) to (1,1);
\draw[bend left, draw=black, thick,dash pattern=on 4pt off 2pt on 1pt off 2pt] (1,0) to (2,0);
\draw[bend left, draw=black, thick,dash pattern=on 4pt off 2pt on 1pt off 2pt] (2,0) to (2,1);
\draw[bend left, draw=black, thick,dash pattern=on 4pt off 2pt on 1pt off 2pt] (2,1) to (1,1);

\draw[bend right, draw=green, thick,dash pattern=on 2pt off 1pt on 1pt off 2pt] (0,1) to (0,2);
\draw[bend left, draw=green, thick,dash pattern=on 2pt off 1pt on 1pt off 2pt] (1,1) to (1,2);
\draw[bend left, draw=green, thick,dash pattern=on 2pt off 1pt on 1pt off 2pt] (0,1) to (1,1);
\draw[bend right, draw=green, thick,dash pattern=on 2pt off 1pt on 1pt off 2pt] (0,2) to (1,2);

\draw[bend right, draw=orange, line width=0.7mm, dotted] (0,0) to (0,2);
\draw[bend left, draw=orange, line width=0.7mm, dotted] (1,0) to (1,2);
\draw[bend right, draw=orange, line width=0.7mm, dotted] (0,0) to (1,0);
\draw[bend left, draw=orange, line width=0.7mm, dotted] (0,2) to (1,2);

\draw[bend right, draw=cyan, line width=0.7mm, dash pattern=on 2pt off 1pt on 1pt off 2pt] (1,0) to (1,2);
\draw[bend left, draw=cyan, line width=0.7mm, dash pattern=on 2pt off 1pt on 1pt off 2pt] (2,0) to (2,2);
\draw[bend right, draw=cyan, line width=0.7mm, dash pattern=on 2pt off 1pt on 1pt off 2pt] (1,0) to (2,0);
\draw[bend left, draw=cyan, line width=0.7mm, dash pattern=on 2pt off 1pt on 1pt off 2pt] (1,2) to (2,2);

\draw[bend left, draw=brown, line width=0.7mm, dashed] (0,1) to (2,1);
\draw[bend right, draw=brown, line width=0.7mm, dashed] (0,2) to (2,2);
\draw[bend left, draw=brown, line width=0.7mm, dashed] (0,1) to (0,2);
\draw[bend right, draw=brown, line width=0.7mm, dashed] (2,1) to (2,2);

\draw[bend left, draw=yellow!80!black, line width=0.7mm,dash pattern=on 4pt off 2pt on 1pt off 2pt] (0,0) to (2,0);
\draw[bend right, draw=yellow!80!black, line width=0.7mm,dash pattern=on 4pt off 2pt on 1pt off 2pt] (0,1) to (2,1);
\draw[bend left, draw=yellow!80!black, line width=0.7mm,dash pattern=on 4pt off 2pt on 1pt off 2pt] (0,0) to (0,1);
\draw[bend right, draw=yellow!80!black, line width=0.7mm,dash pattern=on 4pt off 2pt on 1pt off 2pt] (2,0) to (2,1);

\draw[bend right, draw=pink!80!red, line width=0.7mm] (0,0) to (2,0);
\draw[bend left, draw=pink!80!red, thick,line width=0.7mm] (0,2) to (2,2);
\draw[bend left, draw=pink!80!red, thick,line width=0.7mm] (0,0) to (0,2);
\draw[bend left, draw=pink!80!red, thick,line width=0.7mm] (2,2) to (2,0);

\node (1-1) at (0,0) {};
\node (1-2) at (1,0) {};
\node (1-3) at (2,0) {};

\node (2-1) at (0,1) {};
\node (2-2) at (1,1) {};
\node (2-3) at (2,1) {};

\node (3-1) at (0,2) {};
\node (3-2) at (1,2) {};
\node (3-3) at (2,2) {};
\end{tikzpicture}
    \caption{The Hamming graph $H(2,3)$ (left)  alongside its cycle decomposition into nine $C_4$ cycles (right).}
    \label{fig:H(2,3)}
\end{figure}
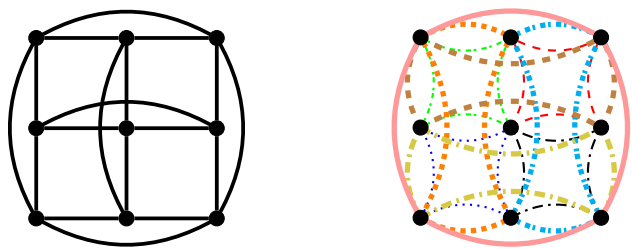

Let $G$ be the Heawood graph, a distance-regular graph with intersection array $\{3, 2, 2; 1, 1, 3\}$. Since $c_2 = 1$ so by Proposition \ref{prop:c_i_are_1} we have $3 \leq q(G) \leq 4$. However, after that we do not have any general technique to apply. Consider the set $\Phi_3(G)$, which contains $28$ polynomials, each expressed as a sum of three square-free monomials of degree $3$. There are $21$ variables in total, with each appearing in exactly $12$ polynomials from $\Phi_3(G)$. If $q(G) = 3$, it would require assigning values $1$ or $-1$ to each variable such that none of the $28$ polynomials evaluates to $3$ or $-3$. Using a brute force computer search we determined  that no such assignment exists. Therefore, $q(G)$ must be $4$.

\begin{table}[ht]
\centering
\begin{tabular}{|r|l|l|l|l|l|}
\hline
$d$ & $n$ & intersection array &  graph name  & $q(G)$ & explanation\\ \hline
$2$ & $9$& $\{4,2;1,2\}$ & Hamming-$H(2,3)$ & $3$ & Prop. \ref{prop:q(G)>t}, $s = 9, j=2$.\\
$2$ & $16$ & $\{5,4;1,2\}$ & Clebsch graph & $2\ (=\dot{q}(G))$ & Prop. \ref{prop:double_cover_drg}, Wells graph.\\ 
$2$ & $16$ & $\{6,3;1,2\}$ & Shrikhande graph & $3$ & Prop. \ref{prop:q(G)>t}, $s= 11, j =2$.\\ 
$2$ & $35$ & $\{16,9;1,8\}$ & folded Johnson $J(8,4)$ & $2\ (=\dot{q}(G))$ & Prop. \ref{prop:double_cover_drg},  $J(8,4)$. \\
$2$ & $64$ & $\{28,15;1,12\}$ & folded halved $8$-cube & $2\ (=\dot{q}(G))$ & Prop. \ref{prop:double_cover_drg}, $\frac{H(8,2)}{2}$. \\
$2$ & $77$ & $\{16,15;1,4\}$ & $M_{22}$ graph &  $3$ & Prop. \ref{cor_triangle_free_drg}.\\
$3$ & $12$ & $\{5,2,1;1,2,5\}$ & Icosahedron graph & $3$ or $4$ & Prop. \ref{prop:q(G)>t}, $s = 15, j = 2$.\\
$3$ & $14$ & $\{3,2,2;1,1,3\}$ & Heawood graph & $4$ & See above. \\ 
$3$ & $14$ & $\{4,3,2;1,2,4\}$ & distance $3$ of Heawood & $2\ (=\dot{q}(G))$ & The graph $S_{14}$ from  \cite{mckee2007integer}.\\  
$3$ & $35$ & $\{4,3,3;1,1,2\}$ & Kneser $K(7,3)$ & $4$ & Prop. \ref{prop:q(G)>t}, $s = 9, j = 3$. \\ 
$3$ & $57$ & $\{6,5,2;1,1,3\}$ & Perkel graph & $3$ or $4$ & Prop. \ref{cor_triangle_free_drg}. \\
$4$ & $28$ & $\{3,2,2,1;1,1,1,2\}$ & Coxeter graph & $5$ & Prop. \ref{prop:q(G)>t}, $s = 21, j = 4$. \\
\hline
\end{tabular}
\vspace{0.1cm}
\caption{Summary of some small distance-regular graphs and their $q(G)$ values.}
\label{table1}
\end{table}

\subsection{Simplicial complexes}\label{sec:simp_comp}
We begin by introducing some basic definitions of abstract simplicial complexes and the associated higher-order Laplacian matrices. For further details and properties, see for e.g., \cite[Section 3.12]{brouwer2011spectra}, \cite[Section 2]{bachoc2019theta}, and \cite[Section 2]{reitz2020higher}. Let $V = \{v_1,v_2,\ldots,v_n\}$ be a finite set. An \emph{abstract simplicial complex}  $\Delta$ is defined as a collection of subsets of $V$ such that, if $A \in \Delta$ and $B \subseteq A$, then $B \in \Delta$. The subsets belonging to $\Delta$ are referred to as the \emph{faces} of $\Delta$. The \emph{dimension of a face} $A \in \Delta$ is given by $|A| - 1$. Therefore, the empty face $\emptyset$ has dimension $-1$, while singleton sets have dimension $0$, and so forth. Let $\Delta_d$ denote the set of all $d$-dimensional faces of $\Delta$. The \emph{dimension of a simplicial complex}, denoted by $\text{dim}(\Delta)$, is the maximum dimension of any face in $\Delta$.

For instance, a graph can be viewed as a simplicial complex of dimension $1$. Another example is the power set $\mathcal{P}(V)$ of a vertex set $V$. Given a graph $G$, the \emph{clique complex}, \emph{independence complex}, and \emph{matching complex} are simplicial complexes associated with $G$, defined over the vertex set $V$. The faces of these complexes correspond to the cliques, independent sets, and matchings of $G$, respectively. In particular, the power set $\mathcal{P}(V)$ is the clique complex of the complete graph on the vertex set $V$.

We fix an ordering on the vertex set $V$ as $v_1 < v_2 < \cdots < v_n$. This ordering induces a corresponding order on each face. Therefore, whenever we denote a face $F = \{x_1, x_2, \dots, x_{d+1}\}$, it is understood that $x_1 < x_2 < \cdots < x_{d+1}$. For $d = 0,1,\ldots,\text{dim}(\Delta)$, consider the matrix $W_d$ whose rows and columns are indexed by $\Delta_{d-1}$ and $\Delta_{d}$, respectively, and the $(H,F)$-entry is 

\begin{align}\label{eqn_W_d_sim_comp}
    W_{d}(H,F) = \begin{cases}
        (-1)^{i-1}, &\text{if } F = \{x_1,x_2,\ldots,x_{d+1}\} \text{ and } H = F\setminus\{x_i\},\\
        0, &\text{otherwise}. \end{cases}
\end{align}
Let us define the \emph{down-Laplacian matrix} as $L_d^{\downarrow} = W_{d}^TW_{d}$ for $d=0,1,\ldots,\text{dim}(\Delta)$. While the \emph{up-Laplacian matrix} is defined as $L_d^{\uparrow} = W_{d+1}W_{d+1}^T$ for $d=0,1,\ldots,\text{dim}(\Delta)-1$. Note that the off-diagonal entries of $L_d^\downarrow$ and $L_d^\uparrow$ are either $0$ or $\pm 1$. Moreover, the diagonal entries of $L_d^\downarrow$ are all equal to $d + 1$, while the diagonal entries of $L_d^\uparrow$ may vary.

\begin{rem} In the literature, one typically begins by considering the vector spaces of all functions from $\Delta_d$ to $\mathbb{R}$ for each $d$, and then defines the boundary and co-boundary maps, with their matrix representations given by $W_d$ and $W_d^T$, respectively. Similarly, the down-Laplacian and up-Laplacian are introduced as self-adjoint, positive semidefinite operators on these function spaces. However, since this paper focuses primarily on matrices, we omit these details and concentrate solely on their matrix representations.    
\end{rem}

For $d=0,1,\ldots,\text{dim}(\Delta)$, we define a simple graph $\mathcal{G}_d^{\downarrow}$, where the vertices represent all $d$-dimensional faces of $\Delta$, and two faces are adjacent if and only if their intersection is a $(d-1)$-dimensional face of $\Delta$. Similarly, for $d=0,1,\ldots,\text{dim}(\Delta)-1$, we define a simple graph $\mathcal{G}_d^{\uparrow}$, where the vertices are all $d$-dimensional faces of $\Delta$, and two faces are adjacent if and only if their union forms a $(d+1)$-dimensional face of $\Delta$. Consequently, $L_d^\downarrow \in \mathcal{S}(\mathcal{G}_d^{\downarrow})$ and $L_d^\uparrow \in \mathcal{S}(\mathcal{G}_d^{\uparrow})$. Furthermore, if we define matrices $\dot{L}_d^\downarrow$ and $\dot{L}_d^\uparrow$, where the diagonal entries are zero and the off-diagonal entries are the same as in $L_d^\downarrow$ and $L_d^\uparrow$, respectively, then $\dot{L}_d^\downarrow \in \mathcal{\dot{S}}(\mathcal{G}_d^{\downarrow})$ and $\dot{L}_d^\uparrow \in \mathcal{\dot{S}}(\mathcal{G}_d^{\uparrow})$. Thus, knowing the spectra of the down-Laplacian and up-Laplacian matrices of a simplicial complex $\Delta$, may allow us to provide upper bounds on the number of distinct eigenvalues, $q$ and $\dot{q}$, of the associated graphs.

For a given graph $G$ of order $n$, viewed as a $1$-dimensional simplicial complex, $\mathcal{G}_0^{\downarrow}$ is the complete graph of order $n$, $\mathcal{G}_0^{\uparrow}$ is the graph $G$ itself, and $\mathcal{G}_1^{\downarrow}$ is the line graph of $G$. The matrix $L_0^{\downarrow}$ is the all-ones matrix, while $L_0^{\uparrow}$ is the standard Laplacian matrix of $G$. For the power set simplicial complex $\mathcal{P}(V)$ with $|V| = n$, it is known that the only distinct eigenvalues for both $L_d^{\downarrow}$ and $L_d^{\uparrow}$ are $0$ and $n$ for all $d \geq 0$ (see, for example, \cite[Example 2.2]{bachoc2019theta} and \cite[Section 3.12]{brouwer2011spectra}). Interestingly, $\mathcal{G}_d^{\downarrow}$ and  $\mathcal{G}_d^{\uparrow}$ correspond to the Johnson graph $J(n, d+1)$. Consequently, it follows from the above discussion that $q(J(n, d+1)) = \dot{q}(J(n, d+1)) = 2$. We consider the Johnson graphs in Section \ref{section_Johnson}. For the sake of completeness and continuity in the arguments for the new results, we compute the eigenvalues of the down-Laplacian and up-Laplacian of the power set simplicial complex using the method of induction in Section \ref{section_Johnson}.

Consider the clique complex $\Delta = \Delta(d,n)$ of the complete $d$-partite graph $K_{d \times n}$. The dimension of $\Delta$, $\text{dim}(\Delta)$, is equal to $d-1$, since the largest cliques in $K_{d\times n}$ have size $d$. Furthermore, there exists a one-to-one correspondence between these largest cliques and the vertex set of the Hamming graph $H(d,n)$. In fact, the graph $\mathcal{G}_{\text{dim}(\Delta)}^{\downarrow}$ is isomorphic to the Hamming graph $H(d, n)$. However, the down Laplacian matrix $L_{\text{dim}(\Delta)}^{\downarrow}$ equals $dI+A(H(d,n))$, and therefore, unlike the Johnson graph, it is not as useful in terms of providing a better upper bound for $\dot{q}(H(d,n))$. Observe that if $n=2$, then the graph $\mathcal{G}_{\text{dim}(\Delta)-1}^{\uparrow}$ is isomorphic to the line graph of $H(d,2)$. In Section \ref{section_Hamming}, we revisit this perspective on the Hamming graph.

\section{The Johnson graphs}\label{section_Johnson}
Let $n,d \in \mathbb{N}$, $n\geq d+1$, and let $[n]:= \{1,2,\ldots,n\}$. Recall that the \emph{Johnson graph} $J(n,d)$ is the graph whose vertices are the $d$-subsets of $[n]$, where two subsets are adjacent when their intersection has size $d-1$. Since $J(n,d)$ is isomorphic to $J(n,n-d)$ by an isomorphism that maps a subset to its complement so we may assume that $n \geq 2d$. However, we do not need this assumption at most instances. We denote the distance-$j$ graph of $J(n,d)$ by $J(n,d,j)$ for $2\leq j\leq d$. Note that two $d$-subsets are adjacent in $J(n,d,j)$ if and only if their intersection has size $d-j$. The Johnson graph $J(n,d)$ is distance-regular of diameter $d$, and has intersection array given by $b_i = (d-i)(n-d-i)$ and $c_i = i^2$ for all $i = 0,1,\ldots,d$ (see \cite[Theorem 9.1.2]{BCN}).

In what follows, for the sake of clarity, we find it more convenient to use another description of Johnson graphs. Which is in terms of binary words instead of subsets. Thus, we introduce some notations. A \emph{binary word} is a finite string of $0$s and $1$s. Let $X(\alpha)$ and $Z(\alpha)$ denote the sets of indices where $1$s and $0$s occur in a binary word $\alpha$, respectively. The \emph{weight} of $\alpha$ is given by the cardinality $|X(\alpha)|$, and the \emph{length} of $\alpha$ is given by $|X(\alpha)|+|Z(\alpha)|$. Let $\mathcal{B}_{n,d}$ be the set of all binary words of length $n$ and weight $d$, and let $\mathcal{B}_n$ be the set of all binary words of length $n$. We fix the lexicographic ordering on $\mathcal{B}_{n,d}$ and $\mathcal{B}_n$. 

Let $\alpha \in \mathcal{B}_{n,d}$. We denote the set of indices of $0$s before the first occurrence of $1$ in $\alpha$ by $Z_0(\alpha)$, and the set of indices of $0$s after the $d$-th occurrence of $1$ by $Z_d(\alpha)$. For $i = 1,\ldots,d-1$, the set of indices of $0$s between the $i$-th and $(i+1)$-th occurrence of $1$ in $\alpha$ is denoted by $Z_i(\alpha)$. Additionally, let $z_i(\alpha) = |Z_i(\alpha)|$ represents the cardinality of each of these sets. Note that $Z(\alpha) = \cup_{i=0}^dZ_{i}(\alpha)$ and $\sum_{i=0}^dz_i(\alpha) = n-d$.

There is one-to-one correspondence between $\mathcal{B}_{n,d}$ and the set of all $d$-subsets of $[n]$, with a binary word $\alpha$ corresponds to the $d$-subset $X(\alpha)$. Let $\alpha = a_1a_2\cdots a_n$ and $\beta = b_1b_2\cdots b_n$ are two elements of $\mathcal{B}_n$. Then the \emph{hamming distance} between them is defined as $d_H(\alpha,\beta) := |\{i\ |a_i\neq b_i\}|$. Notice that $d_{H}(\alpha,\beta) = |X(\alpha) \triangle X(\beta)|$ where $\triangle$ denotes the symmetric difference of two sets. 

Thus, another description of the Johnson graph $J(n,d)$ is the graph whose vertex set is $\mathcal{B}_{n,d}$, and two binary words $\alpha,\beta$ are adjacent if and only if $d_H(\alpha,\beta) = 2$. Suppose $d_H(\alpha,\beta) = 2$, and $i,j$ are the two indices such that $a_i\neq b_i$ and $a_j \neq b_j$. Then we denote the number of $1$s in $\alpha$ and $\beta$ between the $i$-th and $j$-th indices as  
$
 h(\alpha,\beta) := |\{\ell\ | a_{\ell}=b_{\ell} =1, i<\ell<j\}|.   
$
Let us consider the following signed adjacency matrix $A_{n,d}$ of the Johnson graph $J(n,d)$ whose $(\alpha,\beta)$ entry is given by
\begin{align}\label{eqn_A_entrywise}
    A_{n,d}(\alpha,\beta) = \begin{cases}
    (-1)^{h(\alpha,\beta)}, &\text{if } d_H(\alpha,\beta)= 2,\\
    0, &\text{otherwise}.
    \end{cases}
\end{align}
We consider another matrix $W_{n,d}$ whose rows and columns are indexed by $\mathcal{B}_{n,d}$ and $\mathcal{B}_{n,d+1}$, respectively, and the $(\alpha,\beta)$-entry is
\begin{align}\label{eqn_W_entrywise}
    W_{n,d}(\alpha,\beta) = \begin{cases}
        (-1)^{i-1}, &\text{if } X(\beta) = \{x_1,x_2,\ldots,x_{d+1}\} \text{ and } X(\alpha) = X(\beta)\setminus\{x_i\},\\
        0, &\text{otherwise}.
    \end{cases}
\end{align}
Observe that $W_{n,d}$ is identical to the matrix $W_d$ in Equation \ref{eqn_W_d_sim_comp} for the power set simplicial complex $\mathcal{P}(V)$ where $|V| = n$. Additionally, it is important to highlight that the matrix $W_{n,d}$ represents a signed version of specific instances of Wilson matrices. Wilson matrices play a significant role in the Bose-Mesner algebra of Johnson graphs (see, for example, \cite{GMbook}). The following result is one of our main results.

\begin{theorem}\label{Main_thm_1}
    Let $n, d \in \mathbb{N}$, $n\geq d+1$. Then the eigenvalues of $A_{n,d}$ are $-d$ and $n-d$, with multiplicities are $\binom{n-1}{d}$ and $\binom{n-1}{d-1}$. Moreover, the eigenspace for $-d$ is $\text{col}(W_{n,d})$, and for $n-d$ it is $\text{row}(W_{n,d-1})$. In particular, we have $\dot{q}(J(n,d))= q(J(n,d))=2$. 
\end{theorem}
We prove Theorem \ref{Main_thm_1} in Section \ref{sec:reg_Johnson}. 
In that same section, we also construct an additional matrix using $W_{n,d}$ which is used to determine both the minimum rank $mr(J(n,d))$ and the positive semidefinite minimum rank $mr_+(J(n,d))$. An effective method to lower bound the minimum rank is the zero forcing number which was first introduced in \cite{aim2008zero}. The \emph{zero forcing process} on a graph $G$ is defined as follows. Initially, there is a subset $S$ of blue vertices, while all other vertices are white. The standard color change rule dictates that at each step, a blue vertex with exactly one white neighbor will force its white neighbor to become blue. The set $S$ is said to be a \emph{zero forcing set} if, by iteratively applying the color change rule, all of $V(G)$ can be coloured blue. The \emph{zero forcing number} of $G$ is the minimum cardinality of a zero forcing set in $G$, denoted by $Z(G)$. The bound provided by $Z(G)$ is that $|V(G)|-Z(G) \leq mr(G)$ (see \cite[Proposition 2.4]{aim2008zero}). Thus, $Z(G)$ serve as combinatorial parameter that help to bound an algebraic parameter $mr(G)$. The authors in \cite{abiad2024diameter} computed the zero forcing number of the Johnson graph $J(n,d)$ (in a much more general setting). The zero forcing number of Johnson graph is $Z(J(n,d)) = \binom{n}{d}-\binom{n-2}{d-1}$ (see \cite[Corollary 9]{abiad2024diameter}). It generalizes the computation of $Z(J(n,2))$ from \cite{fallat2018compressed}. Thus, for the Johnson graphs we have 
\begin{align}\label{eqn_zero_forcing_mr_msr}
    \binom{n-2}{d-1} \leq mr(J(n,d)) \leq mr_+(J(n,d)).
\end{align}
The following result which we will prove in Section \ref{sec:reg_Johnson} constitutes our second main finding about the Johnson graphs.

\begin{theorem}\label{main_thm_2}
    Let $d\geq 2$ and $n\geq 2d$. Then $mr(J(n,d)) = mr_+(J(n,d))=\binom{n-2}{d-1}$.
\end{theorem}

Let us denote the signed Johnson graph corresponding to the signed matrix $A_{n,d}$ by $\dot{J}(n,d)$. The graph with the same vertex set as that of $J(n,d)$, however, the edge set consists of positive (or negative) edges of $\dot{J}(n,d)$ is denoted by $J_{+}(n,d)$ (respectively by $J_{-}(n,d)$). See Figure \ref{fig_J_(6,3)} for the graph $J_{-}(6,3)$. In the following theorem, we provide the formulas for the positive degree $k_+(\alpha)$ and the negative degree $k_{-}(\alpha)$ of any vertex $\alpha$ in the signed Johnson graph $\dot{J}(n,d)$.

\begin{figure}[ht]
    \centering
    \begin{tikzpicture}[scale=0.5]
    \def\N{20} 
    \def\R{10} 

    \def\words{{"{\tiny{}111000}", "{\tiny{}110100}", "{\tiny{}110010}", "{\tiny{}110001}", "{\tiny{}101100}", "{\tiny{}101010}", "{\tiny{}101001}", 
    "{\tiny{}100110}", "{\tiny{}100101}", "{\tiny{}100011}", "{\tiny{}011100}", "{\tiny{}011010}", "{\tiny{}011001}", 
    "{\tiny{}010110}", "{\tiny{}010101}", "{\tiny{}010011}", "{\tiny{}001110}", "{\tiny{}001101}", "{\tiny{}001011}", "{\tiny{}000111}"}}

    \foreach \i in {1,...,\N} {
        \pgfmathsetmacro\angle{360/\N * (\i - 1)}

        \node[circle,draw] (N\i) at (\angle:\R) {\pgfmathparse{\words[\i-1]}\pgfmathresult};
    }

\foreach \i in {5,6,7} {
    \draw (N1) -- (N\i);}

\foreach \i in {11,8,9} {
    \draw (N2) -- (N\i);}

\foreach \i in {12,14,10} {
    \draw (N3) -- (N\i);}

\foreach \i in {13,15,16} {
    \draw (N4) -- (N\i);}

\foreach \i in {8,9} {
    \draw (N5) -- (N\i);}

\foreach \i in {17,10} {
    \draw (N6) -- (N\i);}

\foreach \i in {18,19} {
    \draw (N7) -- (N\i);}

    \draw (N8) .. controls +(300:1cm) and +(-320:1cm) .. (N10);

\foreach \i in {20} {
    \draw (N9) -- (N\i);}

\foreach \i in {14,15} {
    \draw (N11) -- (N\i);}

\foreach \i in {17,16} {
    \draw (N12) -- (N\i);}

\foreach \i in {18,19} {
    \draw (N13) -- (N\i);}

    \draw (N14) .. controls +(-300:1cm) and +(150:1cm) .. (N16);

\foreach \i in {20} {
    \draw (N15) -- (N\i);}

    \draw (N17) .. controls +(500:1cm) and +(-220:1cm) .. (N19);

    \draw (N18) .. controls +(500:1cm) and +(-220:1cm) .. (N20);
\end{tikzpicture}
    \caption{Graph $J_{-}(6,3)$: a $3$-regular graph with $20$ vertices.}
    \label{fig_J_(6,3)}
\end{figure}
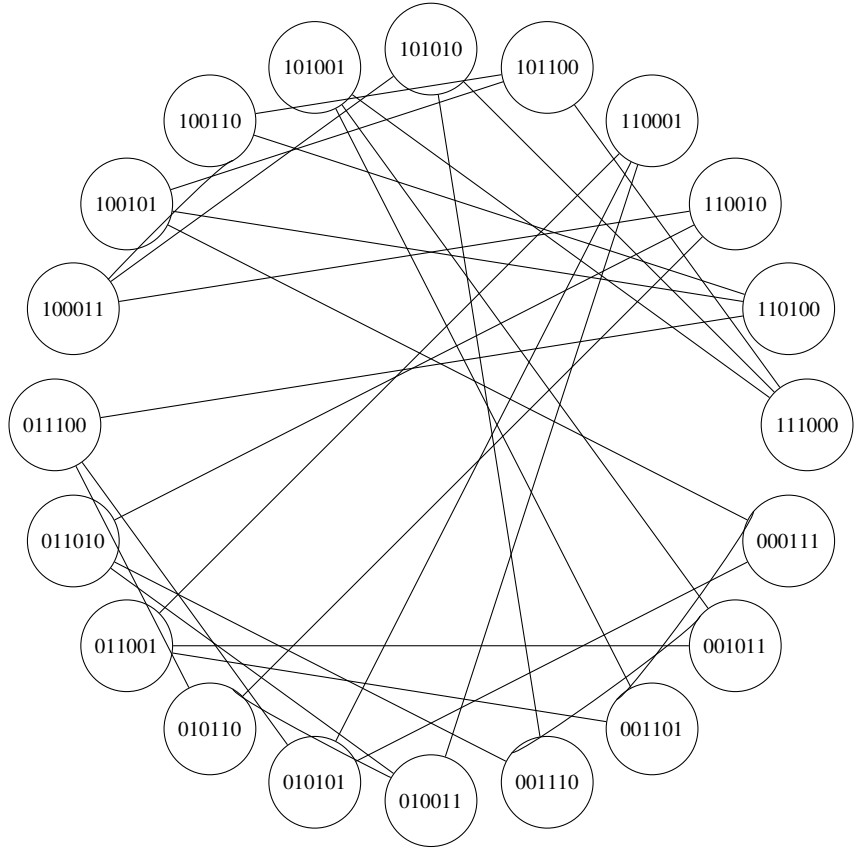

\begin{theorem}\label{prop_pos_and_neg_degree}
     Let $d\geq 1$ and $n\geq d+1$. Suppose $\alpha \in \mathcal{B}_{n,d}$ and $r_{n,d}(\alpha) := \sum_{\ell=0}^{\lfloor(d-1)/2\rfloor}z_{2\ell+1}(\alpha)$. Then 
    \begin{align*}
        k_{+}(\alpha) &= \begin{cases}
            \frac{d+1}{2}(n-d), &\text{if } d \text{ is odd},\\
            \frac{d}{2}(n-d)+r_{n,d}(\alpha) &\text{if } d \text{ is even};
        \end{cases}\\
        k_{-}(\alpha) &= \begin{cases}
            \frac{d-1}{2}(n-d), &\text{if } d \text{ is odd},\\
            \frac{d}{2}(n-d)-r_{n,d}(\alpha), &\text{if } d \text{ is even}.
        \end{cases}
    \end{align*} 
    In particular, the graphs $J_{+}(n,d)$ and $J_{-}(n,d)$ are regular graphs if and only if $d$ is odd. 
\end{theorem}

In the next section, we prove Theorem \ref{prop_pos_and_neg_degree} along with the two previously mentioned theorems. However, in this section, we identify another graph within the Johnson scheme that has exactly two distinct eigenvalues. Specifically, when $n = 3d - 2$, the complement of the Johnson graph has two distinct eigenvalues. To establish the following theorem, we rely on Theorem \ref{thm:sym_scheme_idem}, and utilize the expression for the first eigenmatrix of the Johnson scheme (see Delsarte \cite[p. 48]{delsarte1973algebraic}). 
 For all $0\leq i \leq d$ and $0\leq j\leq d$, the entry $P_j(i)$ of the first eigenmatrix is given by
\begin{equation}\label{eq:john_eig}
    P_j(i) = \sum_{h=0}^{j}(-1)^{j-h}\binom{d-i}{h}\binom{d-h}{j-h}\binom{n-d-i+h}{h}.
\end{equation}
\begin{theorem}\label{thm:complement_J(3d-2,d)}
    Let $d \geq 2$, and let $G$ be the complement graph of the Johnson graph $J(3d-2,d)$. Then $q(G) = 2$. 
\end{theorem}
\begin{proof}
From Equation \ref{eq:john_eig}, we have $P_j(d-1) = (-1)^{j-1}(j-1)\binom{d}{j}$ for $0\leq j \leq d$. Thus, $P_{j}(d-1) \neq 0$ if and only if $j\neq 1$. Since the multiplicity $m_{d-1}$ and the valencies $P_j(0)$ are positive integers, Theorem \ref{thm:sym_scheme_idem} applies with $I = \{d-1\}$ and $J = \{2,3,\ldots,d\}$ to yield the desired result. 
\end{proof}

\subsection{Proofs of the main results}\label{sec:reg_Johnson} 
At the end of this section, we establish Theorem \ref{Main_thm_1}, Theorem \ref{main_thm_2}, and Theorem \ref{prop_pos_and_neg_degree}. However, we first need to introduce several intermediate lemmas. We begin with the following remark, which compiles some key observations about the matrix $W_{n,d}$ for later reference.
\begin{rem}\label{remark_about_non_Zero_entries_in_Wnk}
Every column of $W_{n,d}$ has exactly $d+1$ number of non-zero entries, with $1$ and $-1$ alternating from top to bottom. Therefore, there are $\lceil (d+1)/2\rceil$ number of $1$s, and $\lfloor (d+1)/2 \rfloor$ number of $-1$s in every column of $W_{n,d}$. Note that equation (\ref{eqn_W_entrywise}) provides the entries of $W_{n,d}$ with respect to a fixed column corresponding to $\beta \in \mathcal{B}_{n,d+1}$, and involves modifying $\beta$ to obtain an $\alpha \in \mathcal{B}_{n,d}$. However, it is also necessary to understand the entries of $W_{n,d}$ relative to a fixed row corresponding to $\alpha \in \mathcal{B}_{n,d}$, and the modification of $\alpha$ to obtain an $\beta \in \mathcal{B}_{n,d+1}$. Specifically, if $\beta$ is obtained from $\alpha$ by switching a $0$ corresponding to any of the indices in $Z_i(\alpha)$ to a $1$, then $W_{n,d}(\alpha,\beta) = (-1)^{i-1}$. If more than one $0$ of $\alpha$ are switched to obtain $\beta$, then $W_{n,d}(\alpha,\beta) = 0$. Consequently, in the row of $W_{n,d}$ corresponding to $\alpha$, the number of $1$s is equal to $s_{n,d}(\alpha) := \sum_{\ell=0}^{\lfloor d/2 \rfloor}z_{2\ell}(\alpha)$, while the number of $-1$s is equal to $r_{n,d}(\alpha) = \sum_{\ell=0}^{\lfloor (d-1)/2 \rfloor}z_{2\ell+1}(\alpha)$. Furthermore, each row contains exactly $s_{n,d}(\alpha)+r_{n,d}(\alpha) = n-d$ number of $\pm 1$ entries, and the other entries are zeros. 
\end{rem}

Let $d\geq 1$ and $n\geq d+1$. The following two matrices, both of which have their rows and columns indexed by $\mathcal{B}_{n,d}$, are of significant importance to our results:
\begin{align}
    Q_{n,d} &:= W^T_{n,d-1}W_{n,d-1},\label{eqn_Q_defining}\\
    P_{n,d} &:= W_{n,d} W^T_{n,d}\label{eqn_P_defining}.
\end{align}

Note that $Q_{n,d}$ and $P_{n,d}$ are down-Laplacian and up-Laplacian matrices, respectively, for the power set simplicial complex (see Section \ref{sec:simp_comp}). However, the subsequent discussion will not depend on any prior knowledge of simplicial complexes. We begin by describing the entries of $Q_{n,d}$. The diagonal entries of $Q_{n,d}$ are constant equal to $d$. Since as noted in Remark \ref{remark_about_non_Zero_entries_in_Wnk} each column of $W_{n,d-1}$ contains exactly  $d$ non-zero $\pm 1$ entries, with the remaining entries being zeros. For the off-diagonal entries, consider $\alpha,\beta \in \mathcal{B}_{n,d}$ where $\alpha \neq \beta$. Suppose $|X(\alpha) \cap X(\beta)| = s$. If $0\leq s < d-1$, then no element $\gamma \in \mathcal{B}_{n,d-1}$ exists such that $X(\gamma) \subseteq X(\alpha)$ and $X(\gamma) \subseteq X(\beta)$. This implies that in this case the entry $Q_{n,d}(\alpha,\beta)$ is zero. However, if $s = d-1$, there is exactly one element $\gamma \in \mathcal{B}_{n,d-1}$ such that $X(\gamma) \subseteq X(\alpha)$ and $X(\gamma) \subseteq X(\beta)$. In that case, the Hamming distance $d_H(\alpha,\beta) = 2$, and the entry $Q_{n,k}(\alpha,\beta) = W_{n,d-1}(\gamma,\alpha)W_{n,d-1}(\gamma,\beta)$, which is $\pm 1$. Thus, we need to understand when does the sign of $W_{n,d-1}(\gamma,\alpha)$ and $W_{n,d-1}(\gamma,\beta)$ differ if we modify $\gamma$ to obtain $\alpha$ and $\beta$.  Based on the reasoning provided in Remark \ref{remark_about_non_Zero_entries_in_Wnk}, if $h(\alpha,\beta)$ is odd, the signs of $W_{n,d-1}(\gamma,\alpha)$ and $W_{n,d-1}(\gamma,\beta)$ differ. Otherwise, if $h(\alpha,\beta)$ is even, the signs are the same. Therefore, we have 
\begin{align}\label{eqn_Qnk_entrywise}
    Q_{n,d}(\alpha,\beta) = \begin{cases}
    d, &\text{if } \alpha = \beta,\\
    (-1)^{h(\alpha,\beta)}, &\text{if } d_H(\alpha,\beta)= 2,\\
    0, &\text{otherwise}.
    \end{cases}
\end{align}

In the similar manner, we describe the entries of $P_{n,d}$. The diagonal entries of $P_{n,d}$ are constant equal to $n-d$. Since as noted in Remark \ref{remark_about_non_Zero_entries_in_Wnk} each row of $W_{n,d}$ contains exactly  $n-d$ non-zero $\pm 1$ entries, with the remaining entries being zeros. For the off-diagonal entries, consider $\alpha,\beta \in \mathcal{B}_{n,d}$ where $\alpha \neq \beta$. Suppose $|X(\alpha) \cap X(\beta)| = s$. If $0\leq s < d-1$, then no element $\gamma \in \mathcal{B}_{n,d+1}$ exists such that $X(\alpha) \subseteq X(\gamma)$ and $X(\beta) \subseteq X(\gamma)$. This implies that in this case the entry $P_{n,d}(\alpha,\beta)$ is zero. However, if $s = d-1$, there is exactly one element $\gamma \in \mathcal{B}_{n,d+1}$ such that $X(\alpha) \subseteq X(\gamma)$ and $X(\beta) \subseteq X(\gamma)$. In that case, the Hamming distance $d_H(\alpha,\beta) = 2$, and the entry $P_{n,d}(\alpha,\beta) = W_{n,d}(\alpha,\gamma)W_{n,d}(\beta,\gamma)$, which is $\pm 1$. So here also we need to understand when does the sign of $W_{n,d}(\alpha,\gamma)$ and $W_{n,d}(\beta,\gamma)$ differ if we modify $\alpha$ and $\beta$ to obtain $\gamma$. Based on the similar reasoning provided in Remark \ref{remark_about_non_Zero_entries_in_Wnk}, if $h(\alpha,\beta)$ is even, the signs of $W_{n,d}(\alpha,\gamma)$ and $W_{n,d}(\beta,\gamma)$ differ. Otherwise, if $h(\alpha,\beta)$ is odd, the signs are the same. Therefore, we have 
\begin{align}\label{eqn_Pnk_entrywise}
    P_{n,d}(\alpha,\beta) = \begin{cases}
    n-d, &\text{if } \alpha = \beta,\\
    (-1)^{h(\alpha,\beta)+1}, &\text{if } d_H(\alpha,\beta)= 2,\\
    0, &\text{otherwise}.
    \end{cases}
\end{align}
We immediately obtain from the entry-wise descriptions (\ref{eqn_Qnk_entrywise}) and (\ref{eqn_Pnk_entrywise}) of matrices $Q_{n,d}$ and $P_{n,d}$ that 
\begin{align}\label{eqn_sum_of_P_and_Q}
    Q_{n,d} + P_{n,d} = nI_{\binom{n}{d}}.
\end{align}
Moreover, for the signed adjacency matrix $A_{n,d}$, equation (\ref{eqn_A_entrywise}), we have 
\begin{align}\label{eqn_A_in_terms_of_P}
    A_{n,d} = Q_{n,d}-dI = (n-d)I - P_{n,d}.
\end{align}

The following discussion establish a key feature of the aforementioned families of matrices. Which is that each of them can be partitioned as a block matrix with a recursive structure. This property enables us to apply the method of induction to prove several results. Note that
\begin{align}
   W_{n,0} &= J_{1,n}, \label{eqn_W0n}\\
 W_{d+1,d} &= \begin{bmatrix}
    1 & -1 & 1 & \cdots & (-1)^d
\end{bmatrix}^T.\label{eqn_Wkk+1}
\end{align}
Suppose $d\geq 1$ and $n\geq d+2$. Then we have a block partition of the matrix $W_{n,d}$ by recognizing that all the binary words in $\mathcal{B}_{n,d}$ ending in $1$ correspond to $\mathcal{B}_{n-1,d-1}$, while those ending in $0$ correspond to $\mathcal{B}_{n-1,d}$. Moreover, according to the lexicographic ordering the elements of $\mathcal{B}_{n-1,d-1}$ are before that of $\mathcal{B}_{n-1,d}$. This gives us the following block partition
\begin{align}
    W_{n,d} = \begin{bmatrix}
        W_{n-1,d-1} & O_{\binom{n-1}{d-1},\binom{n-1}{d+1}}\\
        (-1)^dI_{\binom{n-1}{d}} & W_{n-1,d}\\
    \end{bmatrix}. \label{eqn_Wkn}
\end{align} 
Next the equations (\ref{eqn_Q_defining}), (\ref{eqn_P_defining}), (\ref{eqn_sum_of_P_and_Q}), (\ref{eqn_A_in_terms_of_P}), (\ref{eqn_W0n}), and (\ref{eqn_Wkk+1}) implies that
\begin{align}
    P_{n,1} &= nI_{n}-J_{n,n},\label{eqn_P1n}\\
    P_{d+1,d} &= W_{d+1,d}W^T_{d+1,d} = [(-1)^{i+j}]_{0\leq i,j \leq d},\label{eqn_Pkk+1}\\ 
     A_{n,1} &= -I_{n} + J_{n,n},\label{eqn_An1}\\
     A_{d+1,d} &= I_{d+1} - [(-1)^{i+j}]_{0\leq i,j \leq d}.\label{eqn_Akk+1}
\end{align}
Suppose $d\geq 2$ and $n\geq d+2$. Then the following block partition of the matrix $P_{n,d}$ is obtained by using equation (\ref{eqn_Wkn}) as
\begin{align*}
    P_{n,d} &= W_{n,d} W^T_{n,d}\\
    &= \begin{bmatrix}
        W_{n-1,d-1} & O_{\binom{n-1}{d-1},\binom{n-1}{d+1}}\\
        (-1)^dI_{\binom{n-1}{d}} & W_{n-1,d}\\
    \end{bmatrix}\begin{bmatrix}
        W^T_{n-1,d-1} & (-1)^dI_{\binom{n-1}{d}}\\
        O_{\binom{n-1}{d+1},\binom{n-1}{d-1}}
         & W^T_{n-1,d}\\
    \end{bmatrix}\\
    &= \begin{bmatrix}
        W_{n-1,d-1}W^T_{n-1,d-1} & (-1)^dW_{n-1,d-1}\\
        (-1)^dW^T_{n-1,d-1} & I_{\binom{n-1}{d}}+W_{n-1,d}W^T_{n-1,d}\\
    \end{bmatrix}\\
    &= \begin{bmatrix}
        P_{n-1,d-1} & (-1)^dW_{n-1,d-1}\\
        (-1)^dW^T_{n-1,d-1} & I_{\binom{n-1}{d}}+P_{n-1,d}
    \end{bmatrix}. \numberthis \label{eqn_Pkn}
\end{align*}
By using equations (\ref{eqn_A_in_terms_of_P}) and (\ref{eqn_Pkn}) we obtain the following block partition of the matrix $A_{n,d}$ as
\begin{align}\label{eqn_Akn}
    A_{n,d} = \begin{bmatrix}
        A_{n-1,d-1} & (-1)^{d-1}W_{n-1,d-1}\\
        (-1)^{d-1}W_{n-1,d-1}^T & A_{n-1,d}
    \end{bmatrix}.
\end{align}

\begin{lemma}\label{lemma_prod_of_W_with_W}
Let $d\geq 1$ and $n\geq d+1$. Then 
\begin{align*}
W_{n,d-1}W_{n,d} = O_{\binom{n}{d-1},\binom{n}{d+1}}
\end{align*}
\end{lemma}
\begin{proof}
We apply the method of double induction on $d$ and $n$. Let $d=1$ and $n\geq 2$. In this case, since every column of $W_{n,1}$ has exactly one $+1$ and one $-1$ so by using equation (\ref{eqn_W0n}) we have
\begin{align*}
W_{n,0}W_{n,1} &= J_{1,n}W_{n,1} = O_{\binom{n}{0},\binom{n}{2}}.
\end{align*}
Let $n = d+1$ and $d\geq 1$. In this case, we use another induction on $d$. The base case of $d=1$ is already proved. Let us assume that $W_{d,d-2}W_{d,d-1} = O_{\binom{d}{d-2},\binom{d}{d}}$. Thus, by using equations (\ref{eqn_Wkk+1}) and (\ref{eqn_Wkn}) we obtain
\begin{align*}
W_{d+1,d-1}W_{d+1,d} = \begin{bmatrix}
W_{d,d-2} & O_{\binom{d}{d-2},\binom{d}{d}}\\
(-1)^{d-1}I_{\binom{d}{d-1}} & W_{d,d-1}
\end{bmatrix} \begin{bmatrix}
W_{d,d-1}\\
(-1)^d
\end{bmatrix} = \begin{bmatrix}
W_{d,d-2}W_{d,d-1}\\
(-1)^{d-1}W_{d,d-1}+(-1)^dW_{d,d-1}
\end{bmatrix} = O_{\binom{d+1}{d-1},\binom{d+1}{d+1}}.
\end{align*}
Now, let us assume that $W_{n_0,d_0-1}W_{n_0,d_0} = O_{\binom{n_0}{d_0-1},\binom{n_0}{d_0+1}}$ for all $1\leq d_0 <d$ and $d_0+1 \leq n_0 <n$. Then by using the equation (\ref{eqn_Wkn}) we obtain
\begin{align*}
W_{n,d-1}W_{n,d} &= \begin{bmatrix}
        W_{n-1,d-2} & O_{\binom{n-1}{d-2},\binom{n-1}{d}}\\
        (-1)^{d-1}I_{\binom{n-1}{d-1}} & W_{n-1,d-1} 
        \end{bmatrix}
        \begin{bmatrix}
        W_{n-1,d-1} & O_{\binom{n-1}{d-1},\binom{n-1}{d+1}}\\
        (-1)^dI_{\binom{n-1}{d}} & W_{n-1,d}
        \end{bmatrix}\\
        &= \begin{bmatrix}
        W_{n-1,d-2}W_{n-1,d-1} & O_{\binom{n-1}{d-2},\binom{n-1	}{d+1}}\\
         (-1)^{d-1} W_{n-1,d-1} + (-1)^d W_{n-1,d-1}  & W_{n-1,d-1} W_{n-1,d}
        \end{bmatrix}\\
         &= O_{\binom{n}{d-1},\binom{n}{d+1}}. \qedhere
\end{align*}
\end{proof}

\begin{lemma}\label{lemma_rank_of_Wnk}
    Let $n\geq d+1$. Then rank$(W_{n,d}) = \binom{n-1}{d}$.
\end{lemma}
\begin{proof}
    Let us apply the method of double induction on $d$ and $n$. If $d=0$ and $n\geq 1$, then the statement holds from equation (\ref{eqn_W0n}). If $n=d+1$ and $d\geq 1$, then it holds from equation (\ref{eqn_Wkk+1}). Now, let us assume that rank$(W_{n,d-1}) = \binom{n-1}{d-1}$. By Lemma \ref{lemma_prod_of_W_with_W} we have $W_{n,d}^TW_{n,d-1}^T = O_{\binom{n}{d+1},\binom{n}{d-1}}$. That is the row space of $W_{n,d-1}$ is contained in the null space of $W_{n,d}^T$. Therefore, $\binom{n-1}{d-1}\leq \text{nullity}(W_{n,d}^T) = \binom{n}{d}-\text{rank}(W_{n,d})$. Thus, rank$(W_{n,d})\leq \binom{n-1}{d}$. On the other hand, second row of the block partition of $W_{n,d}$, equation (\ref{eqn_Wkn}), implies that rank$(W_{n,d})\geq \binom{n-1}{d}$.
\end{proof}

\begin{lemma}\label{lemma_prod_of_P_and_W}
 Let $d\geq 1$ and $n\geq d+1$. Then $P_{n,d}W_{n,d} = n W_{n,d}$ and $P_{n,d}W_{n,d-1}^T = O_{\binom{n}{d},\binom{n}{d}}$.
\end{lemma}
\begin{proof}
For the first equation we use the method of double induction on $d$ and $n$. Let $d=1$ and $n\geq 2$. Since every column of $W_{n,1}$ has exactly one $+1$ and one $-1$ so by using equation (\ref{eqn_P1n}) we have 
$P_{n,1}W_{n,1} = (nI_{n}-J_{n,n})W_{n,1} = nW_{n,1}$. Let $d\geq 2$ and $n = d+1$. Then by equations (\ref{eqn_Wkk+1}) and (\ref{eqn_Pkk+1}) we obtain 
\begin{align*}
P_{d+1,d}W_{d+1,d} = [(-1)^{i+j}]_{0\leq i,j\leq d}[(-1)^{i}]_{0\leq i \leq d} = (d+1)[(-1)^{i}]_{0\leq i \leq d}  = (d+1)W_{d+1,d}.
\end{align*}
Let us assume that $P_{n_0,d_0}W_{n_0,d_0} = n_0W_{n_0,d_0}$ for all $1\leq d_0 < d$ and $d_0 +1 \leq n_0 <n$. Therefore, by using the equations (\ref{eqn_sum_of_P_and_Q}), (\ref{eqn_Wkn}), and (\ref{eqn_Pkn}) together with Lemma \ref{lemma_prod_of_W_with_W} we obtain
\begin{align*}
P_{n,d}W_{n,d} &= \begin{bmatrix}
        P_{n-1,d-1} & (-1)^{d}W_{n-1,d-1}\\
        (-1)^{d}W^T_{n-1,d-1} & I_{\binom{n-1}{d}}+P_{n-1,d}
    \end{bmatrix} \begin{bmatrix}
        W_{n-1,d-1} & O_{\binom{n-1}{d-1},\binom{n-1}{d+1}}\\
        (-1)^{d}I_{\binom{n-1}{d}} & W_{n-1,d}\\
    \end{bmatrix}\\
    &= \begin{bmatrix}
    P_{n-1,d-1}W_{n-1,d-1} + W_{n-1,d-1} & (-1)^{d}W_{n-1,d-1}W_{n-1,d}\\
    (-1)^{d}(Q_{n-1,d}+P_{n-1,d}+I_{\binom{n-1}{d}}) & W_{n-1,d} + P_{n-1,d}W_{n-1,d}
    \end{bmatrix}\\
    &= \begin{bmatrix}
    (n-1)W_{n-1,d-1} + W_{n-1,d-1} & O_{\binom{n-1}{d-1},\binom{n-1}{d+1}}\\
    (-1)^{d}((n-1)I_{\binom{n-1}{d}}+I_{\binom{n-1}{d}}) & W_{n-1,d} + (n-1)W_{n-1,d}
    \end{bmatrix}\\
    &= n\begin{bmatrix}
        W_{n-1,d-1} & O_{\binom{n-1}{d-1},\binom{n-1}{d+1}}\\
        (-1)^{d}I_{\binom{n-1}{d}} & W_{n-1,d}\\
    \end{bmatrix} = nW_{n,d}.
\end{align*}
Lemma \ref{lemma_prod_of_W_with_W} implies that $P_{n,d}W_{n,d-1}^T = W_{n,d}W_{n,d}^TW_{n,d-1}^T = O_{\binom{n}{d},\binom{n}{d}}$.
\end{proof}
As an immediate corollary of Lemma \ref{lemma_rank_of_Wnk} and Lemma \ref{lemma_prod_of_P_and_W} we have the following result.
\begin{cor}\label{cor_P_is_idempotent}
    Let $d\geq 1$ and $n\geq d+1$. Then the eigenvalues of $P_{n,d}$ are $n$ and $0$, with multiplicities are $\binom{n-1}{d}$ and $\binom{n-1}{d-1}$. Moreover, the eigenspace for $n$ is $\text{col}(W_{n,d})$, and for $0$ it is $\text{row}(W_{n,d-1})$. In particular, we have $P_{n,d}^2 = nP_{n,d}$. \qed
\end{cor}

\noindent Let $d\geq 2$ and $n\geq d+1$. We define the following two families of matrices as follows
\begin{align}
    M_{n,d} &:= \begin{bmatrix}
        P_{n-1,d-1} & W_{n-1,d-1}
    \end{bmatrix},\\
    R_{n,d} &:= M_{n,d}^TM_{n,d}.\label{eqn:Rkn}
\end{align}
The next result is about a block partition and the eigenvalues of the matrix $R_{n,d}$.
\begin{lemma}\label{lemma_Rkn}
Let $d\geq 2$ and $n\geq d+1$. Then we have 
\begin{align*}
        R_{n,d} = \begin{bmatrix}
        (n-1)P_{n-1,d-1} & (n-1)W_{n-1,d-1}\\
        (n-1)W^T_{n-1,d-1}& (n-1)I_{\binom{n-1}{d}}-P_{n-1,d} \end{bmatrix}. 
\end{align*}
Furthermore, the eigenvalues of $R_{n,d}$ are $0$ and $n(n-1)$, with respective multiplicities are $\binom{n}{d}-\binom{n-2}{d-1}$ and $\binom{n-2}{d-1}$.
\end{lemma}
\begin{proof} 
By using equation (\ref{eqn_sum_of_P_and_Q}), Lemma \ref{lemma_prod_of_P_and_W}, and Corollary \ref{cor_P_is_idempotent} we obtain
 \begin{align*}
    R_{n,d} &= \begin{bmatrix}
        P_{n-1,d-1}\\
        W^T_{n-1,d-1}
    \end{bmatrix}\begin{bmatrix}
        P_{n-1,d-1} & W_{n-1,d-1}
    \end{bmatrix}
    = \begin{bmatrix}
        P^2_{n-1,d-1} & P_{n-1,d-1}W_{n-1,d-1}\\
        W^T_{n-1,d-1}P_{n-1,d-1} & W^T_{n-1,d-1}W_{n-1,d-1}
    \end{bmatrix}\\
    &= \begin{bmatrix}
        (n-1)P_{n-1,d-1} & P_{n-1,d-1}W_{n-1,d-1}\\
        W^T_{n-1,d-1}P_{n-1,d-1} & Q_{n-1,d}
    \end{bmatrix}
    = \begin{bmatrix}
        (n-1)P_{n-1,d-1} & P_{n-1,d-1}W_{n-1,d-1}\\
        W^T_{n-1,d-1}P_{n-1,d-1} & (n-1)I_{\binom{n-1}{d}}-P_{n-1,d}
    \end{bmatrix}\\
    &= \begin{bmatrix}
        (n-1)P_{n-1,d-1} & (n-1)W_{n-1,d-1}\\
        (n-1)W^T_{n-1,d-1}& (n-1)I_{\binom{n-1}{d}}-P_{n-1,d}
    \end{bmatrix}.
\end{align*}
    For the eigenvalues of $R_{n,d}$, we use equation (\ref{eqn_P_defining}) and Corollary \ref{cor_P_is_idempotent} to obtain
    \begin{align*}
        M_{n,d}M^T_{n,d} &= \begin{bmatrix}
        P_{n-1,d-1} & W_{n-1,d-1}
    \end{bmatrix} \begin{bmatrix}
        P_{n-1,d-1}\\
        W^T_{n-1,d-1}
    \end{bmatrix}
    = P^2_{n-1,d-1}+W_{n-1,d-1}W^T_{n-1,d-1}  =  P^2_{n-1,d-1}+P_{n-1,d-1}\\
    &= (n-1)P_{n-1,d-1}+P_{n-1,d-1} = nP_{n-1,d-1}.
    \end{align*}
   By Corollary \ref{cor_P_is_idempotent} we have that the eigenvalues of $M_{n,d}M^T_{n,d}$ are $0$ and $n(n-1)$, with respective multiplicities are $\binom{n-2}{d-2}$ and $\binom{n-2}{d-1}$. Therefore, the second statement about the eigenvalues of $R_{n,d}$ follows since $M_{n,d}M^T_{n,d}$ and $R_{n,d} = M^T_{n,d}M_{n,d}$ have the same non-zero eigenvalues.
\end{proof}

In the rest of the section, we provide proofs of Theorem \ref{Main_thm_1}, Theorem \ref{main_thm_2}, and Theorem \ref{prop_pos_and_neg_degree}.
\begin{proof}[Proof of Theorem \ref{Main_thm_1}]
It follows from Corollary \ref{cor_P_is_idempotent} since by equation (\ref{eqn_A_in_terms_of_P}) we have $A_{n,d} = (n-d)I-P_{n,d}$.
\end{proof}

\begin{proof}[Proof of Theorem \ref{main_thm_2}] 
By equations (\ref{eqn_Pkn}) and the first part of Lemma \ref{lemma_Rkn} it is clear that the zero and non-zero pattern of $R_{n,d}$ is same as that of $P_{n,d}$. Therefore, by equation (\ref{eqn_Pnk_entrywise}) we have that $R_{n,d} \in \mathcal{S}(J(n,d))$. From the second part of Lemma \ref{lemma_Rkn} the matrix $R_{n,d}$ is positive semidefinite with $\text{rank}(R_{n,d}) = \binom{n-2}{d-1}$. Thus, the result follows from equation (\ref{eqn_zero_forcing_mr_msr}).
\end{proof}

\begin{proof}[Proof of Theorem \ref{prop_pos_and_neg_degree}]
Since the Johnson graph $J(n,d)$ is $d(n-d)$-regular so we only need to prove the statement for $k_+(\alpha)$. For that we count the number of $1$s in the row of $A_{n,d}$ corresponding to $\alpha$ by using the method of double induction. Let $d = 1$ and $n\geq 2$. Then the statement follows from equation (\ref{eqn_An1}). For $n=d+1$ and $d\geq 1$ we use equation (\ref{eqn_Akk+1}). If $d$ is odd, then each row has of $A_{d+1,d}$ has $\lceil\frac{d}{2}\rceil$ number of $1$s. If $d$ is even, then the number $1$s in each row of $A_{d+1,d}$ alternates between $\frac{d}{2}$ and $\frac{d}{2}-1$ from the top to bottom. That is exactly captured by $r_{d+1,d}(\alpha)$ since it alternates between $0$ and $1$ from the top to bottom. Let us assume that the statement is true for all $1\leq d_0 < d$ and $d_0+1\leq n_0 < n$. Let $\alpha' \in \mathcal{B}_{n-1}$ is the binary word obtained from $\alpha$ by removing its last entry. Let us recall the block partition (\ref{eqn_Akn}) of the matrix $A_{n,d}$ for the remaining proof. We consider two different cases.
\begin{description}
    \item[Case 1] Suppose $d$ is even. If the last entry of $\alpha$ is $1$, then $\alpha' \in \mathcal{B}_{n-1,d-1}$ and $r_{n,d}(\alpha) = r_{n-1,d-1}(\alpha')$. By the induction hypothesis the contribution for number of $1$s from $A_{n-1,d-1}$ is $\frac{d}{2}(n-d)$. The remaining contribution of $r_{n-1,d-1}(\alpha')$ is due to $-W_{n-1,d-1}$ by Remark \ref{remark_about_non_Zero_entries_in_Wnk}. If the last entry of $\alpha$ is $0$, then $\alpha' \in \mathcal{B}_{n-1,d}$ and $r_{n,d}(\alpha) = r_{n-1,d}(\alpha')$. By the induction hypothesis the contribution for number of $1$s from $A_{n-1,d}$ is $\frac{d}{2}(n-d-1)+r_{n-1,d}(\alpha')$. The remaining contribution of $\frac{d}{2}$ is from $-W_{n-1,d-1}^T$ by Remark \ref{remark_about_non_Zero_entries_in_Wnk}.
    \item[Case 2] Suppose $d$ is odd. If the last entry of $\alpha$ is $1$, then $\alpha' \in \mathcal{B}_{n-1,d-1}$ and $r_{n-1,d-1}(\alpha')+s_{n-1,d-1}(\alpha') = n-d$. By the induction hypothesis the contribution for number of $1$s from $A_{n-1,d-1}$ is $\frac{d-1}{2}(n-d)+r_{n-1,d-1}(\alpha')$. The remaining contribution of $s_{n-1,d-1}(\alpha')$ is due to $W_{n-1,d-1}$ by Remark \ref{remark_about_non_Zero_entries_in_Wnk}. Suppose the last entry of $\alpha$ is $0$. By the induction hypothesis the contribution for number of $1$s from $A_{n-1,d}$ is $\frac{d+1}{2}(n-d-1)$. The remaining contribution of $\frac{d+1}{2}$ is from $W_{n-1,d-1}^T$ by Remark \ref{remark_about_non_Zero_entries_in_Wnk}. 
\end{description} 
Hence, this completes the proof.
\end{proof}

\subsection{Ancillary results}\label{sec:ancillary_results}
This section includes several ancillary results based on the main results from the previous section.

\subsubsection{Weighing matrices}\label{sec:weighing}
Recall that a \emph{weighing matrix} of weight $w$ and order $n$ is a square $n\times n$ matrix $A$ over $\{0,-1,1\}$ such that $AA^T = wI_n$.
The following theorem is proved by Gregory \cite{gregory2012spectra}.
\begin{theorem}[Gregory \cite{gregory2012spectra}]\label{thm:Gregory}
    Let $\dot{G} = (G,\sigma)$ be a signed graph. Then $\rho(\dot{G}) \geq \sqrt{k}$ where $k$ is the average degree of $G$. Moreover, equality happens if and only if $G$ is $k$-regular and $A(\dot{G})$ is a symmetric weighing matrix of weight $k$. 
\end{theorem}

Since $J(n,d)$ is $d(n-d)$-regular so as an immediate corollary of Theorem \ref{Main_thm_1} and Theorem \ref{thm:Gregory} we obtain the following proposition. However, it is easy to verify directly by definition as well. 
\begin{proposition}\label{prop_about_Weighing_matrices} 
    Let $d\geq 1$. Then $A_{2d,d}$ is a weighing matrix of weight $d^2$ and order $\binom{2d}{d}$.  \qed
\end{proposition}

Let $G$ be a $k$-regular graph with $n$ vertices, and let $A$ represent its signed adjacency matrix, satisfying $A^2 = k I_n$. Let us define the matrix $B = 
\begin{bmatrix}
A & I_n \\
I_n & -A
\end{bmatrix}$. It follows that $B^2 = (k + 1) I_{2n}$. Furthermore, $B$ is a signed adjacency matrix for the $(k + 1)$-regular graph $H = G \Osq K_2$. Thus, by starting with any $4$-regular graph that has an orthogonal signed adjacency matrix, one can construct a $5$-regular graph with an orthogonal signed adjacency matrix. In \cite{mckee2007integer}, numerous $4$-regular graphs with orthogonal signed adjacency matrix are constructed. Thus, in \cite{belardo2019open}, the natural question of identifying another $5$-regular graph that is not formed through the above method was posed (see \cite[Problem 3.24]{belardo2019open}). In \cite{alon2021unitary}, this question was addressed with the construction of a $5$-regular graph that is not of the above form.

It is worth noting here that $J(2d, d)$ is not a Cartesian product of any graph with $K_2$. One way to illustrate this is by using the eigenvalues. If the distinct eigenvalues of a graph $G$ are $\theta_i$, then all the distinct eigenvalues of the Cartesian product $G \Osq K_2$ are $\theta_i \pm 1$. Therefore, if $k$ is the largest eigenvalue of $G \Osq K_2$, then $k-2$ must also be an eigenvalue. However, using Equation \ref{eq:john_eig}, the eigenvalues of $J(2d, d)$ are given by $(d - i)^2 - i$ for $0 \leq i \leq d$, which do not include $d^2-2$ as one of its eigenvalues.

\subsubsection{Linear ternary codes}\label{sec:tern_codes}
We first introduce some definitions which are necessary to state our results. For a background on coding theory, linear codes, and their applications, see for e.g., \cite{lint1999introduction}.
Let $N$ be a positive integer, and let $q$ be a prime power. A \emph{linear $q$-ary code} $\C$ of \emph{length} $N$ is a linear subspace of the vector space $\mathbb{F}_q^N$. If $q = 3$, then it is called as \emph{linear ternary code}. The \emph{dimension} $|\C|$ of a linear code $\C$ is its dimension as a linear subspace over $\F_q$. The elements of $\C$ are called \emph{codewords}. The \emph{size} of $\C$ is the number of codewords which is equal to $q^{|\C|}$. The \emph{weight} $X(\alpha)$ of a codeword $\alpha$ is the number of non-zero entries of $\alpha$. The \emph{distance} $\text{dist}(\C)$ of a linear code $\C$ is the minimum weight of its non-zero codewords. Technically, it is defined as the minimum hamming distance between any two distinct codewords. However, for linear codes these two definitions are equivalent. The \emph{dual code} $\C^{\perp}$ is the set of all vectors of $\F_q^N$ which are orthogonal to all codewords of $\C$ over $\F_q$. That is, $\C^{\perp} = \{\beta \in \F_q^N\ |\ \alpha\cdot \beta = 0 \pmod{q} \text{ for all }\alpha\in \C\}$. The code $\C$ is called \emph{self-orthogonal} if $\C \subseteq \C^\perp$, and it is called \emph{self-dual} if $\C = \C^{\perp}$. Note that $|\C| + |\C^{\perp}| = N$. 

Let $A$ be a matrix over $\F_q$ with $N$ columns. The rows of $A$ can be viewed as vectors in $\F_q^N$. Consequently, the linear span of these rows generates a linear $q$-ary code, which we denote by $\C_A$. Consider a signed graph $\dot{G}$ of order $N$. The signed adjacency matrix $A = A(\dot{G})$ can be interpreted as a matrix over $\F_3$ by realizing $2$ as $-1$. Thus, we can consider the linear ternary code $\C_A$ generated by $A$ which has length $N$. 

Stani\'{c} \cite{stanic2024linear} studied the linear ternary codes $\C_A, \C_{A-I}$, and $\C_{A+I}$ where $A$ is the signed adjacency matrix of a strongly regular signed graph. A strongly regular signed graph generalizes the concept of strongly regular graphs which is defined in \cite{stanic2019strongly} (see,  \cite{stanic2019strongly,stanic2024linear} for definition). Note that every signed graph with two eigenvalues is inherently strongly regular. Therefore, it follows from Theorem \ref{Main_thm_1} that the signed Johnson graph $\dot{J}(n,d)$ is a strongly regular. The results from \cite{stanic2024linear} are about the dimension and distance of linear ternary codes, including the following theorem (see \cite[Theorem 3.1]{stanic2024linear}), which we apply to the signed adjacency matrix $A_{n,d}$ of the Johnson graph $J(n,d)$.

\begin{theorem}(\cite[Theorem 3.1]{stanic2024linear})\label{thm_Stanic}
    Let $\dot{G}$ be a signed graph of order $N$ and two integral eigenvalues, $\lambda$ of multiplicity $m_{\lambda}$ and $\mu$ of multiplicity $m_{\mu}$. If $A$ is the signed adjacency matrix of $\dot{G}$, then the following statement holds true:
    \begin{enumerate}
        \item If $\lambda\mu = 0 \pmod{3}$ and $\lambda+\mu = 0\pmod{3}$, then $\C_A$ is self-orthogonal with $|\C_A| \leq \min\{m_{\lambda},m_{\mu}\}$. In addition $\C_{A+I} = \C_{A-I} = \mathbb{F}_3^N$.
        \item If $\lambda\mu = 0 \pmod{3}$ and $\lambda+\mu \neq 0\pmod{3}$, then $\C_{A-(\lambda+\mu)I} = \C_A^{\perp}$ and exactly one of equalities $\lambda = 0 \pmod{3}$, $\mu = 0 \pmod{3}$ holds. If $\lambda = 0 \pmod{3}$, then $|\C_A| = m_{\mu}$ and $|\C_{A-(\lambda+\mu)I}| = m_{\lambda}$. 
        \item If $\lambda\mu \neq 0 \pmod{3}$ and $\lambda+\mu = 0\pmod{3}$, then $\C_A = \mathbb{F}_3^N$ and $\C_{A-I} = \C_{A+I}^\perp$ where $\{|\C_{A-I}|,|\C_{A+I}|\} = \{m_{\lambda},m_{\mu}\}$.
        \item If $\lambda\mu \neq 0 \pmod{3}$ and $\lambda+\mu \neq 0\pmod{3}$, then $\C_{A} = \C_{A-(\lambda+\mu)I} = \mathbb{F}_3^N$ and $\C_{A-\lambda I} = \C_{A-\mu I} \subseteq \C_{A-\lambda I}^\perp$, with $|\C_{A-\lambda I}|\leq \min\{m_{\lambda},m_{\mu}\}$.
    \end{enumerate}
\end{theorem}
Note that the proof of \cite[Theorem 3.1]{stanic2024linear} also implies that $\C_{A+(\lambda+\mu)I} = \F_3^N$ in part $(2)$, and explicitly determines the dimensions of $\C_{A+I}$ and $C_{A-I}$ in part $(3)$ of Theorem \ref{thm_Stanic}. Let $A = A_{n,d}$ be the signed adjacency matrix of the Johnson graph $J(n,d)$. By Theorem \ref{Main_thm_1}, the two integral eigenvalues of $A$ are $\lambda = n-d$ of multiplicity $m_{\lambda} = \binom{n-1}{d-1}$, and $\mu = -d$ of multiplicity $m_{\mu} = \binom{n-1}{d}$. Let $\bar{d} = d\pmod{3}$ and $\bar{n} = n\pmod{3}$. In Table \ref{table_for_linear_ternary_code}, by using  Theorem \ref{thm_Stanic}, we list 
the dimensions and the dual linear codes of $\C_A$, $\C_{A-I}$, and $\C_{A+I}$.
\begin{table}[ht]
    \centering
    \begin{tabular}{ |c|c|c|c| }
    \hline
  &  $\bar{n} = 0$ &  $\bar{n} = 1$ & $\bar{n} = 2$ \\
 \hline
$\bar{d} = 0$ & \vtop{\hbox{\strut $|\C_A| \leq \binom{n-1}{d-1}$, $\C_{A} \subseteq \C_A^{\perp}$,}\hbox{\strut $\C_{A-I} = \C_{A+I} = \F_3^{\binom{n}{d}}$}}& \vtop{\hbox{\strut $|\C_{A}| = \binom{n-1}{d-1}$,}\hbox{\strut $\C_{A-I} = \C_A^{\perp}$, $\C_{A+I} = \F_3^{\binom{n}{d}}$}} & \vtop{\hbox{\strut  $|\C_{A}| = \binom{n-1}{d-1}$,}\hbox{\strut $\C_{A+I} = \C_A^{\perp}$, $\C_{A-I} = \F_3^{\binom{n}{d}}$}} \\ 
 \hline
$\bar{d} = 1$ & \vtop{\hbox{\strut $|\C_{A+I}| \leq \binom{n-1}{d-1}$, $\C_{A+I} \subseteq \C_{A+I}^{\perp}$,}\hbox{\strut $\C_{A} = \C_{A-I} = \F_3^{\binom{n}{d}}$}} & \vtop{\hbox{\strut $|\C_{A+I}| = \binom{n-1}{d-1}$,}\hbox{\strut $\C_{A} = \C_{A+I}^{\perp}$, $\C_{A-I} = \F_3^{\binom{n}{d}}$}} & \vtop{\hbox{\strut $|\C_{A+I}| = \binom{n-1}{d-1}$,}\hbox{\strut $\C_{A-I} = \C_{A+I}^{\perp}$, $\C_{A} = \F_3^{\binom{n}{d}}$}} \\ 
 \hline
$\bar{d} = 2$ & \vtop{\hbox{\strut $|\C_{A-I}| \leq \binom{n-1}{d-1}$, $\C_{A-I} \subseteq \C_{A-I}^{\perp}$,}\hbox{\strut $\C_{A} = \C_{A+I} = \F_3^{\binom{n}{d}}$}} & \vtop{\hbox{\strut $|\C_{A-I}| = \binom{n-1}{d-1}$,}\hbox{\strut $\C_{A+I} = \C_{A-I}^{\perp}$, $\C_{A} = \F_3^{\binom{n}{d}}$}} & \vtop{\hbox{\strut $|\C_{A-I}| = \binom{n-1}{d-1}$,}\hbox{\strut $\C_{A} = \C_{A-I}^{\perp}$, $\C_{A+I} = \F_3^{\binom{n}{d}}$}} \\ 
 \hline
\end{tabular}
\vspace{0.1cm}
    \caption{Dimensions and dual codes of the linear ternary codes generated by the signed adjacency matrices of Johnson graphs}
    \label{table_for_linear_ternary_code}
\end{table}

It is important to highlight that we have complete information for the last two columns of  Table \ref{table_for_linear_ternary_code}. However, in the first column, each row contains three self-orthogonal codes with dimensions at most $\binom{n-1}{d-1}$. We think it would be interesting to compute these dimensions explicitly. For smaller values of $n$ and $d$, computational results indicate that all three dimensions in the first column are precisely $\binom{n-1}{d-1}-\binom{n-2}{d-2}$. A more intriguing and challenging problem would be to compute the distances of the non-trivial linear ternary codes $\C_{A}$, $\C_{A-I}$, and $\C_{A+I}$ for the signed adjacency matrix $A = A_{n,d}$ of the Johnson graph $J(n,d)$.  

\subsubsection{Tight frame graphs}\label{sec:tight_frames}
There is a nice frame theoretic perspective of finding a graph $G$ with $q(G) = 2$ (see for e.g., \cite{abdollahi2018frame}). A sequence of vectors $\mathcal{F} = \{f_i\}_{i=1}^\nu$ is a \emph{finite frame} for the standard $m$-dimensional Euclidean space $\mathbb{R}^m$ if there exists constants $0<A\leq B < \infty$ such that
$
    A||x||^2 \leq \sum_{i=1}^n|\langle x,f_i \rangle|^2 \leq B||x||^2
$
for all $x\in \mathbb{R}^m$. The constants $A$ and $B$ are called \emph{frame bounds}. If $A=B$, then the frame is called a \emph{tight frame}. Let $G$ be a graph of order $\nu$. Then a frame $\mathcal{F} = \{f_i\}_{i=1}^\nu$ is called a \emph{frame representation} for graph $G$ if there is a one-to-one correspondence between $\mathcal{F}$ and $V(G)$ such that for any two distinct vertices $u_i,u_j$ we have $\{u_i,u_j\}\in E(G)$ if and only if $\langle f_i,f_j \rangle \neq 0$. If a graph has a frame representation of a tight frame in $\mathbb{R}^m$, then it is called as a \emph{tight frame graph} for $\mathbb{R}^m$. In \cite[Theorem 5.2]{abdollahi2018frame}, it has been proved that $G$ is a tight frame graph for $\mathbb{R}^m$ if and only if there exists a positive semidefinite matrix $M \in \mathcal{S}(G)$ such that $q(M)=2$ and $\text{rank}(M) = m$.

In \cite{furst2020tight}, the authors identified certain line graphs as tight frame graphs, including the line graph of the complete graph $K_n$, that is the Johnson graph  $J(n,2)$. They proved that $J(n,2)$ is a tight frame graph for $\mathbb{R}^{n-2}$, with $n-2$ is the least possible dimension (see \cite[Theorem 4.3]{furst2020tight}). In this paper, we utilize the matrix $R_{n,d}$ (see equation (\ref{eqn:Rkn})) to provide more general result for Johnson graphs. We note in passing that the matrix referenced in the proof of \cite[Theorem 4.3]{furst2020tight} (see \cite[Theorem 3.18]{aim2008zero} and \cite[Theorem 3.1.31]{peters2012positive}) is quite similar to our matrix $R_{n,2}$. The following is an immediate corrollary of Lemma \ref{lemma_Rkn} and Theorem \ref{main_thm_2}.
\begin{cor}\label{cor_Johnson_is_tight_frame}
    Let $d\geq 2$ and $n\geq 2d$. Then the the Johnson graph $J(n,d)$ is a tight frame graph for $\mathbb{R}^{\binom{n-2}{d-1}}$, with $\binom{n-2}{d-1}$ is the least possible dimension. \qed
\end{cor}

\subsubsection{Induced subgraphs of Johnson graphs}\label{sec:sensitivity}
We apply the following technique from \cite{huang2019induced} to prove Proposition \ref{Prop_max_deg_ind_John}, which concerns the maximum degree of certain induced subgraphs of the Johnson graph $J(n,d)$. In \cite{huang2019induced}, a sequence of signed adjacency matrices $A_d$ is constructed for the hypercubes $H(d,2)$,  with eigenvalues $\pm \sqrt{d}$, each having multiplicity $2^{d-1}$. Using Lemma \ref{lem:huang_max_degree} and Lemma \ref{lem:Cauchy_interlace}, Huang proved in \cite{huang2019induced} that any induced subgraph $G$ of $H(d, 2)$ with $2^{d-1} + 1$ vertices satisfies $\Delta(G) \geq \sqrt{d}$ (see \cite[Theorem 1.1]{huang2019induced}). We use this idea in the following result. 

\begin{proposition}\label{Prop_max_deg_ind_John}
    Let $n$ and $d$ are positive integers such that $n\geq 2d$. Suppose $G$ is an arbitrary induced subgraph of the Johnson graph $J(n,d)$ with $r$ vertices. Then the following holds.
    \begin{enumerate}[(i)]
        \item If $r = \binom{n-1}{d}+1$, then $\Delta(G)\geq n-d$.
        \item If $r = \binom{n-1}{d-1}+1$, then $\Delta(G)\geq d$.
    \end{enumerate}
\end{proposition}
\begin{proof} Let $r = \binom{n-1}{d}+1$, and let $A = A_{n,d}$ from Equation \ref{eqn_A_entrywise}. Suppose $B$ is the principal submatrix of $A$ corresponding to the vertices of $G$. By Lemma \ref{lem:Cauchy_interlace} and Theorem \ref{Main_thm_1}, we have $\lambda_1(B) \geq \lambda_{\binom{n-1}{d-1}}(A) = n-d$. Applying Lemma \ref{lem:huang_max_degree}, it follows that $\Delta(G) \geq \lambda_1(B) \geq n-d$. The statement $(ii)$ follows similarly by considering  $A = -A_{n,d}$.
\end{proof}

\section{The Hamming graphs}\label{section_Hamming}
Let $d,n \in \mathbb{N}$, $n\geq 2$, and let $Y_n :=\{0,1,\ldots,n-1\}$. Recall that the \emph{Hamming graph} $H(d,n)$ has the set $Y_n^d$ as its vertex set, where two $d$-tuples are adjacent if and only if they differ in exactly one coordinate. We denote the distance-$j$ graph of $H(d,n)$ by $H(d,n,j)$ for $2\leq j \leq d$. In this graph, two $d$-tuples are adjacent if and only if they differ in exactly $j$ coordinate positions. Additionally, the graphs $H(d,n)$ and $H(d,n,d)$ are the Cartesian and the tensor products of $d$ copies of the complete graph $K_n$, respectively. The Hamming graph $H(d,n)$ is distance-regular of diameter $d$, and has intersection array given by $b_i = (d-i)(n-1), \ \ c_i = i$ for all $i = 0,1,\ldots,d$ (see \cite[Theorem 9.2.1]{BCN}). In this section, we focus on studying $q(H(d,n))$. 

We begin with the case of $n=2$, that is, the hypercubes $H(d,2)$. As previously noted, $\dot{q}(H(d, 2)) = 2$ for all $d \geq 1$, as established in \cite{ahmadi2013minimum} and \cite{huang2019induced}. The proofs in both \cite{ahmadi2013minimum} and \cite{huang2019induced} proceed in a similar manner which is as follows. Given that $H(d,2) = H(d-1,2) \Osq K_2$ for $d\geq 2$, we can apply the construction method outlined in Section \ref{sec:weighing}. Starting with $M_1 = A(K_2)$, we then recursively construct the matrix 
\begin{align}\label{eqn:matrix_for_hypercube}
M_d = \begin{bmatrix}
    M_{d-1} & I\\
    I & -M_{d-1}
\end{bmatrix}
\end{align}
for $d\geq 2$. As a result, we obtain $M_d \in \dot{\mathcal{S}}(H(d,2))$ and $M_d^2 = d I$. Now, we aim to provide a new perspective on this construction. 

Let $\Delta$ be the clique complex of $K_{d \times 2}$. In Section \ref{sec:simp_comp}, we noted that $\mathcal{G}_{\text{dim}(\Delta)}^{\downarrow}$ is isomorphic to $H(d, 2)$ and $\mathcal{G}_{\text{dim}(\Delta)-1}^{\uparrow}$ is isomorphic to the line graph of $H(d, 2)$. We can represent the set of cliques of size $d$ in $K_{d \times 2}$ as $\mathcal{B}_d$, the set of binary words of length $d$. The set of cliques of size $d - 1$, denoted by $\mathcal{B}_d^{*}$, can be viewed as the set of words of length $d$ with symbols from $\{0, 1, *\}$, where exactly one $*$ appears, indicating the part excluded in a $d-1$-clique of $K_{d \times 2}$. We assume $\mathcal{B}_d^{*}$ is ordered lexicographically, with $0 < 1 < *$. Clearly, $|\mathcal{B}_d^{*}| = d2^{d-1}$. As an example, $\mathcal{B}_3^{*} = \{00*, 01*, 0\hspace{-2.5pt}*\hspace{-2.5pt}0, 0\hspace{-2pt}*\hspace{-2.5pt}1, 10*, 11*, 1\hspace{-2.5pt}*\hspace{-2.5pt}0, 1\hspace{-2.5pt}*\hspace{-2.5pt}1, *00, *01, *10, *11\}$.

For $d\geq 1$, we recursively construct two sequences of $\{0,\pm 1\}$-matrices, whose rows are indexed by $\mathcal{B}_d^{*}$ and columns by $\mathcal{B}_d$, where $(\alpha,\beta)$-th entry is  non-zero if and only if the clique corresponding to $\alpha$ is contained in the clique corresponding to $\beta$.  Let $E_1 = \begin{bmatrix}
    1 & 1
\end{bmatrix}$ and $F_1 = \begin{bmatrix}
    -1 & 1
\end{bmatrix}$. For $d\geq 2$, the recursive relations are given by:
\begin{align*}
    E_d = \begin{bmatrix}
        E_{d-1} & O\\
        O & F_{d-1}\\
        I_{2^{d-1}} & I_{2^{d-1}}
    \end{bmatrix} \text{ and }
    F_d = \begin{bmatrix}
        F_{d-1} & O\\
        O & E_{d-1}\\
        -I_{2^{d-1}} & I_{2^{d-1}}
    \end{bmatrix}.
\end{align*}
By induction on $d\geq 1$, we find that the matrix from Equation \ref{eqn:matrix_for_hypercube} satisfies $M_d = E_d^TE_d - dI = -(F_d^TF_d -dI) \in 
\dot{\mathcal{S}}(H(d,2))$. Note that $E_dE_d^T - 2I \in \dot{\mathcal{S}}(\mathcal{G}_{\text{dim}(\Delta)-1}^{\uparrow})$. Since the non-zero eigenvalues of the matrix $E_dE_d^T$ are same as that of the matrix $E_dE_d^T$ so we obtain the following result.
\begin{proposition}
    Let $d\geq 2$, and let $G$ be the line graph of the hypercube $H(d,2)$. Then $q(G) \leq \dot{q}(G) \leq 3$. \qed
\end{proposition}

Now we focus on the graph $H(d,n)$ for $n\geq 3$. In the following theorem, we demonstrate that it may be sufficient to focus only on $H(d, 3)$ for the consideration of remaining Hamming graphs.

\begin{theorem}\label{thm_hamming_induced_graph}
    Let $n\geq 3$, $e\geq d\geq 2$ are positive integers. Suppose that the system of equations $f = 0$ for $f\in \Phi_d(H(d,3))$ has no solution in which all variables are non-zero. Then $q(H(e,n))\geq d+1$.  
\end{theorem}
\begin{proof}
Consider the subset $U \subseteq V(H(e,n))$ 
consisting of all tuples where the first $d$ entries are from $\{0,1,2\}$ and the last $e-d$ entries are $0$. The induced subgraph on $U$ is $H(d,3)$. Any two vertices in $H(d, 3)$ that are at distance $d$ remain at the same distance $d$ when viewed in $H(e, n)$. Moreover, since for the Hamming graphs we have $c_i = i$ so the number of walks of length $d$ between any two vertices at distance $d$ in $H(d, 3)$ remains the same when counted in $H(e, n)$. Therefore, $\Phi_d(H(d, 3)) \subseteq \Phi_d(H(e, n))$. The result follows from Lemma \ref{lem:subgraph_no_zero_free_solution}.
\end{proof}

The following corollary demonstrates that only hypercubes $H(d,2)$ have exactly two distinct eigenvalues among the Hamming graphs $H(d,n)$.
\begin{cor}\label{cor:q(hamming)>3}
    Let $n\geq 3$ and $d\geq 2$. Then $q(H(d,n)) \geq 3$. In particular, we have $q(H(2,n)) = 3$.
\end{cor}
\begin{proof}
    In Section \ref{sec:small_drg}, we showed that the system of equations $f=0$ for $f\in \Phi_2(H(2,3))$ has no solution in which all variables are non-zero. The result now follows from Theorem \ref{thm_hamming_induced_graph}.
\end{proof}

We note that the fact that $q(H(2,n)) = 3$ in Corollary \ref{cor:q(hamming)>3} also follows from \cite[Lemma 3.3]{barrett2023regular}. We suspect that $q(H(d,n)) = d+1$ for all $d\geq 3$ and $n\geq 3$. One approach to proving this is to establish the following Conjecture \ref{conj_about_hamming} and then apply Theorem \ref{thm_hamming_induced_graph}. Therefore, it is worth noting that a polynomial in $\Phi_d(H(d,3))$ corresponds to a $d$-dimensional hypercube $H(d,2)$ that is contained within $H(d,3)$. Conversely, each induced hypercube $H(d,2)$ of $H(d,3)$ corresponds to $2^{d-1}$ distinct polynomials in $\Phi_d(H(d,3))$, one for each antipodal pair. 
\begin{conjecture}\label{conj_about_hamming}
    Let $d\geq 3$. Then the system of equations $f = 0$ for $f\in \Phi_d(H(d,3))$ has no solution in which all variables are non-zero. 
\end{conjecture}

Let us now consider the other graphs in the Hamming scheme. We start with the distance-$d$ graph $H(d,n,d)$ and its complement. 

\begin{theorem}\label{thm_q=2_for_tensor_of_complete_even_graph}
    Let $n_i \neq 4$ are $d$ even numbers for $1\leq i\leq d$. Then $q(K_{n_1}\times K_{n_2}\times \cdots \times K_{n_d}) = 2$. In particular, if $n\neq 4$ an even number, then $q(H(d,n,d))= 2$.  
\end{theorem}
\begin{proof}
By Lemma \ref{prop_from_Bailey}, there exists a matrix $B_i$ with zeros on the diagonals such that $B_i \in S(K_{n_i})$ for all $1\leq i \leq d$. By scaling, we can assume that the distinct eigenvalues of each $B_i$ are $\pm 1$. Consequently, the distinct eigenvalues of the matix $B = B_1 \otimes B_2 \otimes \cdots \otimes B_d$ are $\pm 1$. Now, Lemma \ref{prop_from_Bjorkman} implies that $B \in \mathcal{S}(K_{n_1}\times K_{n_2}\times \cdots \times K_{n_d})$, establishing the first statement. The second statement follows directly from the first, since $H(d,n,d)$ is the tensor product $K_{n}\times K_{n}\times \cdots \times K_{n}$.
\end{proof}

To establish our last two theorems, we use Theorem \ref{thm:sym_scheme_idem}, and utilize the expression for the first eigenmatrix of the Hamming scheme (see Delsarte \cite[p.\ 39]{delsarte1973algebraic}). For all $0\leq i\leq d$ and $0\leq j\leq d$, the entry $P_{j}(i)$ of the first eigenmatrix of the Hamming scheme is given by
\begin{align}\label{eq:hamming_eigenmatrix}
    P_{j}(i) &= \sum_{h=0}^j (-1)^h(n-1)^{j-h}\binom{i}{h}\binom{d-i}{j-h}.
\end{align}
Moreover, the multiplicity $m_i = (n-1)^i\binom{d}{i}$ for all $0\leq i \leq d$ (see \cite[Theorem 9.2.1]{BCN}). 

\begin{theorem}\label{thm_q=2_for_complement_of_H(d,q,d)}
    Let $n\geq 3$ and $d$ is odd. Then $q\left(\overline{H(d,n,d)}\right) = 2$. 
\end{theorem}
\begin{proof}
By using Equation \ref{eq:hamming_eigenmatrix}, we obtain that $P_j(d) = (-1)^j\binom{d}{j}$ and $P_j(0) = (n-1)^j\binom{d}{j}$ for all $0\leq j \leq d$. Thus, $\frac{m_0P_j(0)}{P_j(0)} + \frac{m_dP_j(d)}{P_j(0)} = 1 + (-1)^j(n-1)^{d-j}$ for all $0\leq j \leq d$. If 
$n\geq 3$ and $d$ is odd, then $1 + (-1)^j(n-1)^{d-j} \neq 0$ exactly when $0\leq j\leq d-1$. Therefore, the statement follows by applying Theorem \ref{thm:sym_scheme_idem} with $I = \{0,d\}$ and $J = \{1,2,\ldots,d-1\}$.    
\end{proof}

We now turn our attention to specific graphs in the Hamming scheme $H(d,2)$. Let $d\geq 3$, $0\leq j\leq d$, and $t \in \{0,1,2\}$. We need to evaluate the sum given by
\begin{align}\label{eqn_binomial_sum_for_hypercubes}
    \zeta(d,j,t)\ := \sum_{\substack{i=0 \\ i\ \equiv\ t\hspace{-7pt}\pmod{3}}}^d \binom{d}{i}\sum_{h=0}^{j}(-1)^h\binom{i}{h}\binom{d-i}{j-h}.
\end{align}
Let $\omega\neq 1$ be a cube root of unity. Then
\begin{align*}
    \zeta(d,0,0) &= \sum_{\substack{i=0 \\ i\ \equiv\ 0\hspace{-7pt}\pmod{3}}}^d \binom{d}{i} = \frac{(1+1)^d+(1+\omega)^d+(1+\omega^2)^d}{3} = \frac{2^d+(-1)^d(\omega^{2d}+\omega^d)}{3},\\
    \zeta(d,0,1) &= \sum_{\substack{i=0 \\ i\ \equiv\ 1\hspace{-7pt}\pmod{3}}}^d \binom{d}{i} = \frac{(1+1)^d+\omega^2(1+\omega)^d+\omega(1+\omega^2)^d}{3}= \frac{2^d+(-1)^d(\omega^{2d+2}+\omega^{d+1})}{3},\\
    \zeta(d,0,2) &= \sum_{\substack{i=0 \\ i\ \equiv\ 2\hspace{-7pt}\pmod{3}}}^d \binom{d}{i} = \frac{(1+1)^d+\omega(1+\omega)^d+\omega^2(1+\omega^2)^d}{3} = \frac{2^d+(-1)^d(\omega^{2d+1}+\omega^{d+2})}{3}.
\end{align*}
For $j\geq 1$, by using Pascal's binomial identity $\binom{n}{k} = \binom{n-1}{k-1} + \binom{n-1}{k}$, we obtain the following recursion: 
\begin{align*}
    \zeta(d,j,t) = \zeta(d-1,j,t)+\zeta(d-1,j-1,t)+\zeta(d-1,j,t-1)-\zeta(d-1,j-1,t-1).
\end{align*}
Note that in the above recursion $t-1$ is modulo $3$. Using the equations above and applying the method of induction, we arrive at the following lemma. To maintain brevity, we omit the standard, yet lengthy, proof.

\begin{lemma}\label{lem_binomial_sum_for_hypercubes}
    Let $d\geq 3$, $0\leq j\leq d$, and $t \in \{0,1,2\}$. Suppose $d = 3r+s$ where $r\geq 1$ and $s \in \{-2,0,2\}$. Let us define $\kappa(d,j) = (-3)^{\left\lceil\frac{j}{2}\right\rceil-1}\binom{d}{j}$. Then, the following statements hold: 
    \begin{enumerate}
        \item If $(s,t) \in \{(-2,1),(2,0),(0,2)\}$, then $\zeta(d,0,t) = (2^d-(-1)^d)/3$ and $\zeta(d,j,t) = (-1)^r\kappa(d,j)$ for all $j\geq 1$.
        \item If $(s,t) \in \{(-2,0),(2,2),(0,1)\}$, then $\zeta(d,0,t) = (2^d-(-1)^d)/3$ and $\zeta(d,j,t) = (-1)^{r+j}\kappa(d,j)$ for all $j\geq 1$.
        \item If $(s,t) \in \{(-2,2),(2,1),(0,0)\}$ and $j$ is even, then $\zeta(d,0,t) = (2^d+(-1)^d 2)/3$ and $\zeta(d,j,t) = (-1)^{r+1}2\kappa(d,j)$ for all $j\geq 2$.
        \item If $(s,t) \in \{(-2,2),(2,1),(0,0)\}$ and $j$ is odd, then $\zeta(d,j,t) = 0$. \qed
    \end{enumerate}
\end{lemma}

We use Equation \ref{eq:hamming_eigenmatrix}, Lemma \ref{lem_binomial_sum_for_hypercubes} and Theorem \ref{thm:sym_scheme_idem} in the following theorem to prove that certain graphs in the Hamming scheme $H(d,2)$ has two distinct eigenvalues. 

\begin{theorem}\label{thm_q=2_for_hypercube_scheme}
    Let $d\geq 3$. Suppose $d = 3r+s \geq 3$ where $r\geq 1$. Then the following statements hold:
    \begin{enumerate}
        \item If $r$ is even and $s\in \{0,1,2\}$, then $q\left(\overline{H(d,2)}\right) = 2$.
        \item If $r$ is odd and $s\in \{0,1,2\}$, then $q\left(\overline{H(d,2,2)}\right) = 2$.
        \item If $r$ is odd and $s\in \{-2,-1,0\}$, then $q\left(\overline{H(d,2)\cup H(d,2,2)}\right) = 2$.
    \end{enumerate}
\end{theorem}
\begin{proof}
    Define the sets $I_t := \{i\ |\ 0\leq i \leq d,\ i \equiv t\hspace{-3pt} \pmod{3}\}$ for $t\in \{0,1,2\}$. From Equation \ref{eq:hamming_eigenmatrix}, it follows that $\sum_{i\in I_t}\frac{m_iP_{j}(i)}{P_j(0)} = \frac{\zeta(d,j,t)}{\binom{d}{j}}$ for any $0\leq j\leq d$. In what follows, we explicitly list the sets $I$ and $J$ to apply Theorem \ref{thm:sym_scheme_idem}.
    \begin{enumerate}
        \item For $s \in \{0, 1, 2\}$, the corresponding $I$ sets are given by $I_1 \cup \{0\}$, $I_0 \setminus \{0\}$, and $I_2 \cup \{0\}$, respectively, and the set $J$ is defined as $J = [d]/\{1\}$. Now, use the part $(2)$ of Lemma \ref{lem_binomial_sum_for_hypercubes}.
        \item For $s \in \{0, 1, 2\}$, the corresponding $I$ sets are given by $I_1 \cup \{0\}$, $I_0 \setminus \{0\}$, and $I_2 \cup \{0\}$, respectively, and the set $J$ is defined as $J =[d]/\{2\}$. Now, use the part $(2)$ of Lemma \ref{lem_binomial_sum_for_hypercubes}.
        \item For $s \in \{-2, -1, 0\}$, the corresponding $I$ sets are given by $I_1 \cup \{0\}$, $I_0 \setminus \{0\}$, and $I_2 \cup \{0\}$, respectively, and the set $J$ is defined as $J =[d]/\{1,2\}$. Now, use the part $(1)$ of Lemma \ref{lem_binomial_sum_for_hypercubes}.
    \end{enumerate}
    Hence, the proof is complete.
\end{proof}

\section{Summary and further questions}\label{sec:summary}
In summation, our aim in this work was to present a careful study of the minimum number of distinct eigenvalues (or $q(G)$) for strongly-regular and distance-regular graphs $G$. The parameter $q(G)$ has become one of the most studied parameters under the umbrella of the inverse eigenvalue problem for graphs. A particularly well-studied aspect concerning this graph parameter is the case when $q(G)=2$, where a formal characterization is still unresolved. Furthermore, research on the parameter $q$ while imposing certain structural constraints on a graph (for e.g., regular, bipartite, edge density, etc.) is not only a natural progression but is also beneficial to the broader inverse eigenvalue problem for graphs community. To this end, we have analyzed a wealth of strongly-regular and distance-regular graphs and in many cases, have computed the minimum number of distinct eigenvalues allowed by such graphs. 

Our work is essentially divided into three parts. We begin our investigation by signaling out such graphs with specific properties, including graphs derived from association schemes, certain distance regular graphs on a small number of vertices and conclude this topic with a reflection on simplicial complexes, which play a role in later sections. The next phase of our work concentrates on the family of Johnson graphs, $J(n,d)$. Our main observation in this section is that $q(J(n,d))=\dot{q}(J(n,d))=2$, but we also consider the minimum rank for such graphs also. In addition, we also lay out a myriad of related interesting advances and discussion involving weighing matrices, linear ternary codes, and tight frame graphs. Finally we consider the class of Hamming graphs, $H(d,n)$. For such graphs we verify cases where $q=2$ and when $q$ is necessarily larger than two. We also include an intriguing discussion when the complements of such graphs have $q$ equal to two. We conclude by presenting several questions that we believe deserve further investigation.

\textbf{Questions related to strongly-regular graphs:} A strongly-regular graph $G$ is called \emph{primitive} if both $G$ and its complement are connected. The Hamming $H(2,n)$ is strongly-regular with $q=3$, and its complement $H(2,n,2)$, has $q=2$ for even $n\geq 6$ (see Theorem \ref{thm_q=2_for_tensor_of_complete_even_graph}). Both the $5$-cycle and the Hamming $H(2,3)$ are self-complementary with $q=3$. The question of whether a primitive strongly-regular graph $G$ exists with $q(G)=q(\overline{G})=2$ remains open.
 
Let $t\equiv 3 \pmod{4}$ be a prime power. The Paley graph $P(t)$ is the graph whose vertices are the elements of the finite field $\mathbb{F}_t$ and where two vertices $a,b$ are adjacent if and only if $a-b$ is a non-zero square in $\mathbb{F}_t$. The Paley graph $P(t)$ is self-complementary strongly-regular graph. The graphs $P(5)$ and $P(9)$ are $C_5$ and $H(2,3)$, respectively. A open question is whether $q(P(t))$ is $2$ or $3$ for $t\geq 13$.
 
There are seven known examples of triangle-free strongly regular graphs: the $5$-cycle, the Petersen graph, the Clebsch graph, the Hoffman-Singleton graph, the Gewirtz graph, the $77$-graph, and the Higman-Sims graph. For the $5$-cycle, the Petersen graph, and the Hoffman-Singleton graph, it is known that $q = 3$, as for all these graphs, $c_2 = \mu = 1$. For the Clebsch graph, $q = 2$, and for the 77-graph, $q = 3$ (see Table \ref{table1}). The open question remains: to find the $q$-values for the Gewirtz and the Higman-Sims graphs?

The three Chang graphs are strongly-regular graphs, with the same parameters as that of Johnson $J(8,2)$ (see, for e.g., \cite[Section 9.2]{brouwer2011spectra}). We think it would be interesting to determine the $q$-values for the Chang graphs. 

\textbf{Questions related to Johnson graphs:} The general question is to determine the $q$-values of other graphs in the Johnson scheme. Specifically, we have shown that $q(\overline{J(3d - 2, d)}) = 2$, but determining $q(\overline{J(n, d)})$ for general $n$ and  $d$ remains unresolved.

An important family of graphs is the \emph{Kneser graph} $K(n, d)$, which is the distance-$d$ graph of the Johnson graph $J(n, d)$. The Kneser graph $K(n, d)$ is triangle-free if and only if $2d < n < 3d$. Hence, if $\binom{n}{d}$ is odd and $2d < n < 3d$, then $q(K(n, d)) \geq 3$. Moreover, it is known that $K(2d + 1, d)$, the \emph{odd graph}, is a distance-regular graph of diameter $d$ with $c_2 = 1$, so $q(K(2d+1, d)) \geq 3$. This raises the broader question of determining $q(K(n, d))$ for general $n$ and $d$.  

The structural and spectral properties of the Johnson graph $J(n, d)$ have been studied extensively from various perspectives. Therefore, it is also valuable to investigate the graphs $J_{+}(n, d)$, $J_{-}(n, d)$, and the 2-lift associated with $\dot{J}(n, d)$ (see Section \ref{section_Johnson} for definitions) from different viewpoints.

\textbf{Questions related to Hamming graphs:}
We have already proposed Conjecture \ref{conj_about_hamming}. The next general question is to determine the $q$-values of other graphs in the Hamming scheme. Note that $q(H(3, 2)) = q(K_4 \square K_2) = 2$ follows from \cite[Corollary 6.8]{ahmadi2013minimum}. However, this result is not covered by Theorem \ref{thm_q=2_for_hypercube_scheme}. A natural question that arises is whether $q = 2$ for the complements of all hypercube graphs $H(d, 2)$. Additionally, another intriguing question is whether $q = 2$ holds for some of the remaining cases not addressed by Theorem \ref{thm_q=2_for_tensor_of_complete_even_graph} and Theorem \ref{thm_q=2_for_complement_of_H(d,q,d)}.

\section*{Acknowledgements}
This project started in January of 2024 as part of the Discrete Mathematics Research Group at the University. The authors thank the other members of this group: Karen Meagher, Seyed Ahmad Mojallal, Shahla Nassersar, Venkata Pantangi for their thoughtful discussions. H.\ Gupta would like to express gratitude to Ferdinand Ihringer, Jack Koolen, and Akihiro Munemasa for their insightful discussions regarding this work during the ``Graphs and Groups, Complexity and Convexity (G2C2) 2024'' conference, held at Hebei Normal University, Shijiazhuang, China.

S.M.\ Fallat was supported in part by an NSERC Discovery Research Grant, Application No.: RGPIN-2019-03934. The work of the PIMS Postdoctoral Fellow H.\ Gupta leading to this publication was supported in part by the Pacific Institute for the Mathematical Sciences. A.\ Herman was supported in part by an NSERC Discovery Development Grant, Application No.: DDG-2023-00014.

\end{document}